\documentclass[11pt,a4paper]{amsart}
\pdfoutput=1
\usepackage{amssymb}
\usepackage{amsmath}
\usepackage{mathtools}
\usepackage{amsthm}
\usepackage{amsfonts}
\usepackage{bbm} 
\usepackage{dutchcal} 
\usepackage{tikz-cd} 
\usepackage{todonotes} 
\usepackage{graphicx} 
\usepackage[toc,page]{appendix} 
\usepackage{kpfonts} 
\usepackage[mathscr]{euscript} 
\definecolor{winered}{rgb}{0.5,0,0}

\usepackage[margin=2.5cm]{geometry} 

\usepackage[style = alphabetic]{biblatex}
\renewbibmacro{in:}{}
\usepackage[citecolor = winered, linkcolor = winered, menucolor = winered, colorlinks = true]{hyperref}

\usepackage[foot]{amsaddr}

\newtheoremstyle{theoremdd}
{\topsep}{\topsep}{\upshape}{0pt}{\bfseries}{.}{ }{\thmname{#1}\thmnumber{ #2}\thmnote{ (#3)}}

\theoremstyle{definition}
\newtheorem{Th}{Theorem}[section]
\newtheorem{Lemma}[Th]{Lemma}
\newtheorem{Cor}[Th]{Corollary}
\newtheorem{Prop}[Th]{Proposition}
\newtheorem{Def}[Th]{Definition}

\newtheorem{Rem}[Th]{Remark}

\newtheorem{Ex}[Th]{Example}

\newcommand{\Map}{\text{Map}}

\newcommand{\R}{\mathbb{R}}
\newcommand{\Z}{\mathbb{Z}}

\renewcommand{\H}{\mathbb{H}}
\newcommand{\cH}{\check{\H}}

\renewcommand{\u}{\underline}

\newcommand{\colim}{\text{colim}}
\newcommand{\ncat}{\mathsf} 
\newcommand{\cat}{\mathcal} 
\newcommand{\pre}{\ncat{Pre}}
\newcommand{\spre}{\ncat{sPre}}

\newcommand{\cart}{\ncat{Cart}}

\newcommand{\site}{\mathscr{C}}

\newcommand{\cl}{\text{cl}}
\newcommand{\B}{\textbf{B}}
\newcommand{\DK}{\ncat{DK}}
\newcommand{\mc}{\text{mc}}
\newcommand{\MapCh}{\u{\text{Map}_{\ncat{Ch}}}}
\newcommand{\Disc}{\text{Disc}}
\newcommand{\Sing}{\text{Sing}}
\newcommand{\tot}{\ncat{tot} \,}

\newcommand{\dR}{\text{dR}}
\newcommand{\conn}{\text{conn}}
\newcommand{\im}{\text{im} \,}

\makeatletter
\newtheorem*{rep@theorem}{\rep@title}
\newcommand{\newreptheorem}[2]{%
\newenvironment{rep#1}[1]{%
 \def\rep@title{#2 \ref{##1}}%
 \begin{rep@theorem}}%
 {\end{rep@theorem}}}
\makeatother

\newreptheorem{theorem}{Theorem}

\newreptheorem{lemma}{Lemma}

\newreptheorem{cor}{Corollary}
\newreptheorem{prop}{Proposition}

\usetikzlibrary{calc}
\usetikzlibrary{decorations.pathmorphing}

\tikzset{curve/.style={settings={#1},to path={(\tikztostart)
    .. controls ($(\tikztostart)!\pv{pos}!(\tikztotarget)!\pv{height}!270:(\tikztotarget)$)
    and ($(\tikztostart)!1-\pv{pos}!(\tikztotarget)!\pv{height}!270:(\tikztotarget)$)
    .. (\tikztotarget)\tikztonodes}},
    settings/.code={\tikzset{quiver/.cd,#1}
        \def\pv##1{\pgfkeysvalueof{/tikz/quiver/##1}}},
    quiver/.cd,pos/.initial=0.35,height/.initial=0}

\tikzset{tail reversed/.code={\pgfsetarrowsstart{tikzcd to}}}
\tikzset{2tail/.code={\pgfsetarrowsstart{Implies[reversed]}}}
\tikzset{2tail reversed/.code={\pgfsetarrowsstart{Implies}}}

\title{The Diffeological \v{C}ech-de Rham Obstruction}

\author{Emilio Minichiello}
\address{CUNY Graduate Center, Email address:  \href{mailto:eminichiello@gradcenter.cuny.edu}{eminichiello@gradcenter.cuny.edu}}

\begin{document}

\maketitle

\begin{abstract}
Using higher topos theory, we explore the obstruction to the \v{C}ech-de Rham map being an isomorphism in each degree for diffeological spaces. In degree 1, we obtain an exact sequence which interprets Iglesias-Zemmour's construction \cite{iglesias2023vcech} in $\infty$-stack cohomology. We obtain new exact sequences in all higher degrees. These exact sequences are constructed using homotopy pullback diagrams that include the $\infty$-stack classifying higher $\R$-bundle gerbes with connection. We also obtain a conceptual and succinct proof that the $\infty$-stack cohomology of the irrational torus $T_K$ for $K \subset \R$ a diffeologically discrete subgroup, agrees with the group cohomology of $K$ with values in $\R$. Finally, for a Lie group $G$, we prove that the groupoid of diffeological principal $G$-bundles with connection one obtains via higher topos theory is equivalent to the groupoid of diffeological principal $G$-bundles with connection defined in \cite{waldorf2012transgression}. 

\end{abstract}
\tableofcontents

\section{Introduction}

Classical differential geometry involves the study of finite dimensional smooth manifolds. As a theory, it has many achievements. One of its most celebrated is the \textbf{\v{C}ech-de Rham Theorem}, more commonly known as the de Rham Theorem\footnote{We call it the \v{C}ech-de Rham Theorem because some authors use ``the de Rham Theorem" to refer to the isomorphism between de Rham cohomology and singular cohomology.}. The \v{C}ech-de Rham Theorem, proven in 1931 by de Rham \cite{de1931analysis}, states that if $M$ is a finite dimensional smooth manifold, then there is an isomorphism
\begin{equation}
H^k_{\dR}(M) \cong \check{H}^k(M, \R^\delta),    
\end{equation}
where $H^k_{\dR}(M)$ denotes the de Rham cohomology of $M$, and $\check{H}^k(M, \R^\delta)$ denotes the \v{C}ech cohomology of $M$ with values in $\R^\delta$, the constant sheaf on the discrete group of real numbers. There are many good textbook accounts of the \v{C}ech-de Rham Theorem, such as \cite[Chapter II]{bott1982differential} and \cite[Chapter 9]{gallier2022homology}. The de Rham cohomology of a finite dimensional smooth manifold $M$ is constructed using its smooth structure, but the \v{C}ech-de Rham Theorem shows that the de Rham cohomology of $M$ is independent of this smooth structure and depends only on the topology of $M$.

Diffeology is a modern framework for differential geometry whose main objects of study are diffeological spaces, encompassing smooth manifolds, orbifolds, and mapping spaces. The category of diffeological spaces is better behaved than the category of finite dimensional smooth manifolds, indeed the category of diffeological spaces is complete, cocomplete and cartesian closed \cite{iglesias2013diffeology}. This makes diffeological spaces attractive to geometers who study spaces that are not finite dimensional smooth manifolds. However, this generalization comes at the cost of losing many of the theorems and constructions of classical differential geometry\footnote{Many of these theorems are lost because not all diffeological spaces have partitions of unity, a crucial ingredient to many theorems in differential geometry.}. Much contemporary work has gone into extending these constructions and theorems to diffeological spaces. The textbook \cite{iglesias2013diffeology} by Iglesias-Zemmour has in particular pushed the theory quite far, defining differential forms, de Rham cohomology, singular cohomology, fiber bundles, and smooth homotopy groups of diffeological spaces amongst many other contributions.

In \cite{iglesias85}, Patrick Iglesias-Zemmour proved that the \v{C}ech-de Rham Theorem does not hold in general for diffeological spaces. Interestingly, this result was written as a preprint in French in the late 80s and was only recently published in English as \cite{iglesias2023vcech}. Furthermore Iglesias-Zemmour obtained an exact sequence
\begin{equation*}
    0 \to H^1_{\dR}(X) \to \check{H}^1_{PIZ}(X,\R^\delta) \to {}^d E^{1,0}_2(X) \xrightarrow{c_1} H^2_{\dR}(X) \to \check{H}^2_{PIZ}(X, \R^\delta)
\end{equation*}
which is a receptacle for the obstruction to the \v{C}ech-de Rham Theorem. The group ${}^d E^{1,0}_2(X)$ is the subgroup of the group of isomorphism classes of diffeological principal $\R$-bundles that admit a connection. If this group is trivial, as it is for all finite dimensional smooth manifolds, then $H^1_{\dR}(X) \cong \check{H}^1_{PIZ}(X)$. However, the situation for higher degrees is not addressed in \cite{iglesias2023vcech}. Iglesias-Zemmour writes ``We must acknowledge that the geometrical natures of the higher obstructions
of the De Rham theorem still remain uninterpreted. It would be certainly
interesting to pursue this matter further" \cite[Page 2]{iglesias2023vcech}. In this paper, we obtain such an interpretation of the higher obstructions. 

In \cite{minichiello2022diffeological}, we introduced a generalization of \v{C}ech cohomology for diffeological spaces that we call \textbf{$\infty$-stack cohomology}. If $X$ is a diffeological space and $A$ is a diffeological abelian group, then $\check{H}^k_\infty(X, A)$ denotes the $k$th $\infty$-stack cohomology of $X$ with values in $A$.

Currently, there are four definitions of \v{C}ech cohomology for diffeological spaces in the literature. They are Iglesias-Zemmour's cohomology from \cite{iglesias2023vcech}, which we call PIZ cohomology, there is $\infty$-stack cohomology \cite{minichiello2022diffeological}, there is Krepski-Watts-Wolbert cohomology \cite{krepski2021sheaves} and there is Ahmadi's cohomology \cite{ahmadi2023diffeological}. In \cite[Section 5.3]{minichiello2022diffeological}, the first three \v{C}ech cohomologies were compared, and some relationships deduced, but it is currently unknown if any of the above cohomology theories agree in general.

This paper is a sequel to \cite{minichiello2022diffeological}, where we explored the connection between diffeological spaces and higher topos theory to study diffeological principal $G$-bundles. When $G$ is a diffeological group, not necessarily abelian, it is still possible to define $\infty$-stack cohomology in degree $1$, $\check{H}^1_{\infty}(X, G)$. We proved \cite[Corollary 6.9]{minichiello2022diffeological} that degree $1$ $\infty$-stack cohomology is in bijection with isomorphism classes of diffeological principal $G$-bundles over $X$. In fact, we obtained a much stronger result by showing that the nerve of the groupoid of diffeological principal $G$-bundles is weak equivalent to the nerve of the category of $G$-principal $\infty$-bundles on $X$ \cite[Theorem 6.8] {minichiello2022diffeological}.

In this paper, we study two cases where the tools of higher topos theory help us better understand diffeological spaces. The first case is studying the $\infty$-stack cohomology of the irrational torus. The irrational torus was the first example of a nontrivial diffeological space with trivial underlying topology, see \cite{iglesias2020torus}. In \cite{iglesias2023vcech}, Iglesias-Zemmour proved that if $K \subset \R$ is a diffeologically discrete subgroup, then the PIZ cohomology of the irrational torus $T_K = \R / K$ is isomorphic to the group cohomology of $K$ with values in $\R$. However, his proof of this, \cite[Section II]{iglesias2023vcech}, is computational. In Section \ref{section irrational torus}, we prove 

\begin{reptheorem}{th infinity stack cohomology of irrational torus is group cohomology}
There is an isomorphism
\begin{equation*}
    \check{H}^n_\infty(T_K, \R^\delta) \cong H^n_\text{grp}(K, \R^\delta)
\end{equation*}
of abelian groups, for every $n \geq 0$, where $\R^\delta$ denotes the discrete group of real numbers, and where $H^n_{\text{grp}}(K, \R^\delta)$ denotes the group cohomology of $K$ with coefficients in $\R^\delta$.
\end{reptheorem}

Theorem \ref{th infinity stack cohomology of irrational torus is group cohomology} supports the conjecture that PIZ cohomology and $\infty$-stack cohomology agree. The proof of Theorem \ref{th infinity stack cohomology of irrational torus is group cohomology} is short and conceptual. It uses the shape functor $\int$, much beloved by higher differential geometers \cite{berwickevans2022classifying}, \cite{bunk2022r}, \cite{schreiber2013differential}, \cite{clough2023homotopy}, \cite{carchedi2015homotopy}, in a crucial way, reducing the $\infty$-stack cohomology of $T_K$ to the singular cohomology of the classifying space $\B K$. This demonstrates the advantage of using $\infty$-stack cohomology to study diffeological spaces.

The second case, which makes up the bulk of the paper, is to use $\infty$-stack cohomology, and more generally the framework of higher topos theory, to study the diffeological \v{C}ech-de Rham obstruction. First we obtain a homotopy pullback diagram of $\infty$-stacks.

\begin{reptheorem}{th hopullback diagram}
For every $k \geq 1$, there exists a commutative diagram of $\infty$-stacks of the following form
\begin{equation}
\begin{tikzcd}
	{*} & {\B^k \R^\delta} & {*} & {*} \\
	{*} & {\B^k_\nabla \R} & {\Omega^{k+1}_\cl} & {\Omega^{k+1}} \\
	& {\B^k \R} & {\B^k \Omega^1_{\cl}} & {\Omega^{1 \leq \bullet \leq k+1}} \\
	& {*} & {\B^{k+1} \R^\delta} & {\B^{k+1}_\nabla \R}
	\arrow[from=1-1, to=2-1]
	\arrow[from=1-1, to=1-2]
	\arrow[from=2-1, to=2-2]
	\arrow[from=1-2, to=2-2]
	\arrow[from=2-2, to=2-3]
	\arrow[from=1-2, to=1-3]
	\arrow[from=1-3, to=2-3]
	\arrow[from=2-2, to=3-2]
	\arrow[from=2-3, to=2-4]
	\arrow[from=1-4, to=2-4]
	\arrow[from=1-3, to=1-4]
	\arrow[from=2-4, to=3-4]
	\arrow[from=2-3, to=3-3]
	\arrow[from=3-2, to=3-3]
	\arrow[from=3-3, to=3-4]
	\arrow[from=4-3, to=4-4]
	\arrow[from=3-4, to=4-4]
	\arrow[from=3-2, to=4-2]
	\arrow[from=4-2, to=4-3]
	\arrow[from=3-3, to=4-3]
\end{tikzcd}
\end{equation}
furthermore every commutative square in this diagram is a homotopy pullback square in the \v{C}ech model structure on simplicial presheaves over $\cart$.
\end{reptheorem}

Such diagrams are often used in higher category-theoretic treatments of differential cohomology, see \cite{schreiber2013differential}, \cite{amabel2021differential}, \cite{Myers2021}. One can think of an $\infty$-stack as a classifying object for a mathematical structure, such as diffeological principal $G$-bundles, as in \cite{minichiello2022diffeological}. Thus the above diagrams can be thought of as tight relationships between the corresponding mathematical structures.

Of particular interest is the $\infty$-stack $\B^k_\nabla \R$. This is the $\infty$-stack which classifies diffeological \textbf{$\R$-bundle $(k-1)$-gerbes with connection}. Cohomology with values in this $\infty$-stack is called the $k$th pure differential cohomology in \cite{Myers2021}. From Theorem \ref{th hopullback diagram} we are immediately able to obtain the following result. 

\begin{repcor}{cor prePIZ exact sequence}
For every diffeological space $X$, there is an exact sequence of vector spaces
\begin{equation} \label{eq degree k prePIZ exact sequence}
   0 \to \check{H}^k_{\infty}(X, \R^\delta) \to \check{H}^k_{\infty, \nabla}(X, \R) \to \Omega^{k+1}_{\cl}(X) \to \check{H}^{k+1}_{\infty}(X, \R^\delta).
\end{equation}
\end{repcor}

Near the completion of this paper, we learned that an analogous exact sequence was also obtained in \cite[Page 27]{Myers2021} using completely different methods in the framework of homotopy type theory. The above exact sequence allows us to compute the pure differential cohomology of the irrational torus.

\begin{reptheorem}{th bundle gerbes of irrational torus}
Let $T_\alpha$ denote the irrational torus, then
\begin{equation}
    \check{H}^k_{\infty, \nabla}(T_\alpha, R) \cong \begin{cases}
        \R^2, & k = 1,\\
        \R, & k = 2,\\
        0, & k > 2.
    \end{cases}
\end{equation}
\end{reptheorem}

While Corollary \ref{cor prePIZ exact sequence} is useful for computations with the irrational torus, it is desirable to have an exact sequence including de Rham cohomology rather than closed forms. This is obtained in the following result.

\begin{reptheorem}{th degree k PIZ exact sequence}
Given a diffeological space $X$ and $k \geq 1$, the sequence of vector spaces
\begin{equation*}
     \check{H}^k_{\infty}(X, \R^\delta) \to \check{H}^k_{\conn}(X, \R) \to H^{k+1}_{\dR}(X) \to \check{H}^{k+1}_{\infty}(X, \R^\delta)
\end{equation*}
is exact.
\end{reptheorem}

When $k = 1$, we obtain an additional piece to this exact sequence.

\begin{reptheorem}{th degree 1 PIZ exact sequence}
Given a diffeological space $X$, the sequence of vector spaces
\begin{equation} \label{eq degree 1 piz exact sequence}
0 \to H^1_{\dR}(X) \to \check{H}^1_{\infty}(X, \R^\delta) \to \check{H}^1_{\conn}(X, \R) \to H^2_{\dR}(X) \to \check{H}^2_{\infty}(X, \R^\delta)
\end{equation}
is exact.
\end{reptheorem}

The above exact sequence is exactly analogous to the exact sequence obtained by Iglesias-Zemmour in \cite{iglesias2023vcech}.

In Appendix \ref{section diff bundles with connection} we turn to the study of connections for diffeological principal bundles. This theory is still in its infancy, and there are a few references that give varying definitions of diffeological connections \cite[Section 8.32]{iglesias2013diffeology}, \cite[Section 3]{waldorf2012transgression}, \cite[Section 4]{magnot2017diffeology}. The theory of $\infty$-stacks provides another definition. Let $G$ be a Lie group, and $U$ a cartesian space. Then let $\Omega^1(U, \mathfrak{g})//G$ denote the groupoid whose objects are differential $1$-forms $\omega \in \Omega^1(U, \mathfrak{g})$, where $\mathfrak{g}$ denotes the Lie algebra of $G$, and where there is a morphism $g: \omega \to \omega'$ if there exists a smooth map $g: U \to G$ such that
\begin{equation*}
    \omega' = \text{Ad}_g^{-1}(\omega) + g^* \mc(G)
\end{equation*}
where $\mc(G)$ denotes the Maurer-Cartan form of $G$. Taking the nerve of this groupoid, and letting $U$ vary defines an $\infty$-stack $\Omega^1(-,\mathfrak{g})//G$, which amongst others has been studied in \cite{fiorenza2011cech}, \cite{freed2013chern}. We connect this notion of connection to that given in \cite[Definition 3.2.1]{waldorf2012transgression} in the following result.

\begin{reptheorem}{th waldorf and cocycle with connections iso}
Given a diffeological space $X$ and a Lie group $G$, the functor
\begin{equation}
    \ncat{Cons}_\nabla : \ncat{Coc}_{\nabla}(X,G) \to \ncat{Wal}_G(X),
\end{equation}
is an equivalence of groupoids, where $\ncat{Coc}_{\nabla}(X,G)$ is the groupoid whose objects are maps $QX \to \Omega^1(-,\mathfrak{g})//G$, where $QX$ is a cofibrant replacement of $X$ in the projective model structure on simplicial presheaves, and $\ncat{Wal}_G(X)$ is the groupoid of diffeological principal $G$-bundles with connection as defined in \cite[Definition 3.2.1]{waldorf2012transgression}.
\end{reptheorem}

To compute $\infty$-stack cohomology, one needs a workable model of the derived mapping space $\R \text{Hom}(X,A)$, when $X$ is a diffeological space and $A$ is a presheaf of chain complexes. In Appendix \ref{section totalization}, we obtain such a model, which reduces many computations with $\infty$-stacks to manipulations with double complexes. As a corollary, we obtain a simple and direct proof of the following well known folklore result.

\begin{repprop}{prop holim of cosimplicial chain complex}
Let $C$ be a cosimplicial chain complex, then
\begin{equation}
    \text{holim}_{n \in \ncat{\Delta}} C^n \simeq \tot C,
\end{equation}
where we are computing the homotopy limit in the category of chain complexes equipped with the projective model structure, and $\tot C$ denotes the total complex of $C$.
\end{repprop}

The paper is organized as follows. In Section \ref{section smooth sheaves and diffeological spaces}, we introduce diffeological spaces and place them in the context of sheaf theory. In Section \ref{section simplicial presheaves}, we introduce simplicial presheaves, show how diffeological spaces embed into simplicial presheaves, and introduce the shape functor. In Section \ref{section irrational torus}, we prove that the $\infty$-stack cohomology of the irrational torus $T_K$ is isomorphic to the group cohomology of $K$ with values in $\R$. In Section \ref{section dold kan}, we introduce the Dold-Kan correspondence, which is a core tool we use for the rest of the paper. In Section \ref{section examples of infinity stacks}, we introduce the main $\infty$-stacks that will be used in the paper, and compute various examples of $\infty$-stack cohomology. In Section \ref{section cech de Rham obstruction}, we prove the main results of this paper, Theorem \ref{th degree k PIZ exact sequence} and Theorem \ref{th degree 1 PIZ exact sequence}. In Appendix \ref{section diff bundles with connection}, we prove that our notion of diffeological principal $G$-bundles with connection using $\infty$-stacks agrees with Waldorf's \cite{waldorf2012transgression}. In Appendix \ref{section totalization}, we prove a technical result allowing us to easily compute $\infty$-stack cohomology when the coefficient $\infty$-stack comes from a presheaf of chain complexes. In Appendix \ref{section proof} we prove Theorem \ref{th hopullback diagram}.

\section{Smooth Sheaves and Diffeological Spaces} \label{section smooth sheaves and diffeological spaces}

In this section we briefly describe diffeological spaces and their connection to sheaves on $\cart$. See \cite[Section 2]{minichiello2022diffeological} for more details.

\begin{Def} \label{def cart, good covers, parametrizations}
Let $M$ be a finite dimensional smooth manifold\footnote{We will assume throughout this paper that manifolds are Hausdorff and paracompact.}. We say a collection of subsets $\mathcal{U} = \{U_i \subseteq M \}_{i \in I}$ is an \textbf{open cover} if each $U_i$ is an open subset of $M$, and $\bigcup_{i \in I} U_i = M$. If $U$ is a finite dimensional smooth manifold diffeomorphic to $\R^n$ for some $n \in \mathbb{N}$, we call $U$ a \textbf{cartesian space}. We call $\mathcal{U} = \{U_i \subseteq M \}$ a \textbf{cartesian open cover} of a manifold $M$ if it is an open cover of $M$ and every $U_i$ is a cartesian space. We say that $\mathcal{U}$ is a \textbf{good open cover} if it is a cartesian open cover, and further every finite non-empty intersection $U_{i_0 \dots i_k} = U_{i_0} \cap \dots \cap U_{i_k}$ is a cartesian space. 

Let $\ncat{Man}$ denote the category whose objects are finite dimensional smooth manifolds and whose morphisms are smooth maps. Let $\cart$ denote the full subcategory whose objects are cartesian spaces. Given a set $X$, let $\ncat{Param}(X)$ denote the set of \textbf{parametrizations} of $X$, namely the collection of set functions $p: U \to X$, where $U \in \ncat{Cart}$.
\end{Def}

\begin{Def} \label{def diffeological space}
A \textbf{diffeology} on a set $X$, consists of a collection $\mathscr{D}$ of parametrizations $p: U \to X$ satisfying the following three axioms:
\begin{enumerate}
    \item $\mathscr{D}$ contains all points $\R^0 \to X$,
    \item If $p: U \to X$ belongs to $\mathscr{D}$, and $f: V \to U$ is a smooth map, then $pf: V \to X$ belongs to $\mathscr{D}$, and
    \item If $\{U_i \subseteq U \}_{i \in I}$ is a good open cover of a cartesian space $U$, and $p: U \to X$ is a parametrization such that $p|_{U_i} : U_i \to X$ belongs to $\mathscr{D}$ for every $i \in I$, then $p \in \mathscr{D}$.
\end{enumerate}
A set $X$ equipped with a diffeology $\mathscr{D}$ is called a \textbf{diffeological space}. Parametrizations that belong to a diffeology are called \textbf{plots}. We say a set function $f: X \to Y$ between diffeological spaces is \textbf{smooth} if for every plot $p: U \to X$ in $\mathscr{D}_X$, the composition $pf: U \to Y$ belongs to $\mathscr{D}_Y$. We often denote the set of smooth maps from $X$ to $Y$ by $C^\infty(X,Y)$. Let $\ncat{Diff}$ denote the category of diffeological spaces.
\end{Def}

Every manifold $M$ is canonically a diffeological space by considering the set of parametrizations $p : U \to M$ that are smooth in the classical sense. This gives a diffeology on $M$, called the \textbf{manifold diffeology}. One can show \cite[Chapter 4]{iglesias2013diffeology} that the manifold diffeology defines a fully faithful functor $\ncat{Man} \hookrightarrow \ncat{Diff}$.

Diffeology extends many constructions and concepts from classical differential geometry to diffeological spaces, such as the theory of bundles.

\begin{Def} \label{def subduction}
We say that a map $\pi : X \to Y$ of diffeological spaces is a \textbf{subduction} if it is surjective, and for every plot $p: U \to Y$, there exists a good open cover $\{U_i \subseteq U \}$, and plots $p_i: U_i \to X$ making the following diagram commute
\begin{equation}
    \begin{tikzcd}
	{U_i} & X \\
	{U} & Y
	\arrow["\pi", from=1-2, to=2-2]
	\arrow["{p_i}", from=1-1, to=1-2]
	\arrow[hook, from=1-1, to=2-1]
	\arrow["{p}"', from=2-1, to=2-2]
\end{tikzcd}
\end{equation}
\end{Def}

\begin{Def}
A \textbf{diffeological group} is a group $G$ equipped with a diffeology such that the multiplication map $m: G \times G \to G$, and inverse map $i: G \to G$ are smooth. A right \textbf{diffeological group action} of a diffeological group $G$ on a diffeological space $X$ is a smooth map $\rho: X \times G \to X$ such that $\rho(x, e_G) = x$, and $\rho(\rho(x,g), h) = \rho(x, gh)$, where $e_G$ denotes the identity element of $G$.
\end{Def}

\begin{Def}
Let $G$ be a diffeological group, and $P$ be a diffeological right $G$-space. A map $\pi: P \to X$ of diffeological spaces is a \textbf{diffeological principal $G$-bundle} if:
\begin{enumerate}
    \item the map $\pi: P \to X$ is a subduction, and
    \item the map $\text{act}: P \times G \to P \times_X P$ defined by $(p,g) \mapsto (p, p \cdot g)$, which we call the \textbf{action map} is a diffeomorphism.
\end{enumerate}
\end{Def}

A map of diffeological principal $G$-bundles $P \to P'$ over $X$ is a diagram
\begin{equation*}
    \begin{tikzcd}
	P && {P'} \\
	& X
	\arrow["{\pi'}", from=1-3, to=2-2]
	\arrow["\pi"', from=1-1, to=2-2]
	\arrow["f", from=1-1, to=1-3]
\end{tikzcd}
\end{equation*}
where $f$ is a $G$-equivariant smooth map. A diffeological principal $G$-bundle $P$ is said to be \textbf{trivial} if there exists an isomorphism $\varphi: X \times G \to P$, called a \textbf{trivialization}, where $\text{pr}_1: X \times G \to X$ is the product bundle. Let $\ncat{DiffPrin}_G(X)$ denote the category of diffeological principal $G$-bundles over a diffeological space $X$.

In \cite{minichiello2022diffeological}, we proved that diffeological principal bundles can be classified using cocycles in a way reminiscent of classical differential geometry. However, rather than using cocycles defined over an open cover, we use cocycles defined on plots. Let $\ncat{Plot}(X)$ denote the category whose objects are plots $p: U \to X$ of $X$ and whose morphisms $f: p \to p'$ are smooth maps $f: U \to U'$ such that $p'f = p$. 

\begin{Def} \label{def cocycle}
Given a diffeological space $X$ and a diffeological group $G$, call a collection $g = \{g_{f_0} \}$ of smooth maps $g_{f_0} : U_{p_1} \to G$ indexed by maps of plots $f_0: U_{p_1} \to U_{p_0}$ of $X$ a \textbf{$G$-cocycle} if for every pair of composable plot maps of $X$
\begin{equation*}
    U_{p_2} \xrightarrow{f_1} U_{p_1} \xrightarrow{f_0} U_{p_0}
\end{equation*}
it follows that
\begin{equation} \label{eqn diff cocycle condition}
    g_{f_0 f_1} = (g_{f_0} \circ f_1) \cdot g_{f_1}.
\end{equation}
We call (\ref{eqn diff cocycle condition}) the \textbf{diffeological $G$-cocycle condition}. 

Given two $G$-cocycles, $g,g'$, we say a collection $h = \{h_{p_0}\}$ of smooth maps $h_{p_0}: U_{p_0} \to G$ indexed by plots of $X$ is a \textbf{morphism of $G$-cocycles} $h : g \to g'$ if for every map $f_0 :U_{p_1} \to U_{p_0}$ of plots of $X$, it follows that
\begin{equation} \label{eqn map of diff cocycles}
    g'_{f_0} \cdot h_{p_1} = (h_{p_0} \circ f_0) \cdot g_{f_0}.
\end{equation}
\end{Def}

Given a diffeological space $X$ and a $G$-cocycle $g$ on $X$, we can construct a diffeological principal $G$-bundle $\pi : P \to X$, by taking the quotient
\begin{equation}
    P = \left( \coprod_{p_0 \in \ncat{Plot}(X)} U_{p_0} \times G \right) /{\sim}
\end{equation}
where $\sim$ is the smallest equivalence relation such that $(x_{p_1}, k_1) \sim (x_{p_0}, k_0)$ if there exists a map $f_0: U_{p_1} \to U_{p_0}$ of plots such that $f_0(x_{p_1}) = x_{p_0}$ and $k_0 = g_{f_0}(x_{p_1}) \cdot k_1$. We let $\pi = \ncat{Cons}(g)$, short for construction. In fact, this construction defines a functor from the category $\ncat{Coc}(X,G)$ of $G$-cocycles on $X$ to the category of diffeological principal $G$-bundles.

\begin{Th}[{\cite[Theorem 3.15]{minichiello2022diffeological}}]\label{th diff cocycle theorem}
Given a diffeological space $X$ and a diffeological group $G$, the functor 
\begin{equation}
    \ncat{Cons}: \ncat{Coc}(X,G) \to \ncat{DiffPrin}_G(X)
\end{equation}
is an equivalence of groupoids.
\end{Th}

While extending the classical theory, there are constructions one can do with diffeological spaces that are not available to smooth manifolds:
\begin{enumerate}
    \item Given a diffeological space $X$, and a subset $A \xhookrightarrow{i} X$ a subset. Then consider the set of parametrizations $p : U \to A$ such that $ip : U \to X$ is a plot of $X$. This collection is a diffeology, called the \textbf{subspace diffeology} on $A$,
    \item Given a diffeological space $X$ and an equivalence relation $\sim$ on $X$, let $\pi: X \to X/{\sim}$ denote the resulting quotient function on sets. Consider the set of parametrizations $p : U \to X/{\sim}$ such that there exists a good open cover $\{U_i \subseteq U \}$ and plots $p_i: U_i \to X$ making the following diagram commute
\begin{equation*}
        \begin{tikzcd}
	{U_i} & X \\
	{U} & X/{\sim}
	\arrow["\pi", from=1-2, to=2-2]
	\arrow["{p_i}", from=1-1, to=1-2]
	\arrow[hook, from=1-1, to=2-1]
	\arrow["{p}"', from=2-1, to=2-2]
\end{tikzcd}
\end{equation*}
This forms a diffeology on $X/{\sim}$, called the \textbf{quotient diffeology},
\item Given a pair $X$ and $Y$ of diffeological spaces, the set of parametrizations $p : U \to X \times Y$ such that the composites $\pi_1 \circ p$ and $\pi_2 \circ p$ are plots of $X$ and $Y$ respectively, forms a diffeology, called the \textbf{product diffeology},
\item Given diffeological spaces $X$ and $Y$, the set of parametrizations $p: U \to C^\infty(X, Y)$ such that the transposed function $p^\#: U \times X \to Y$ is a smooth map is a diffeology, called the \textbf{functional diffeology}.
\end{enumerate}

These constructions make the category of diffeological spaces considerably better than the category of finite dimensional smooth manifolds, as shown in Corollary \ref{cor diff is a quasitopos}.

Diffeological spaces inherit this nice structure from the category of smooth sheaves.

\begin{Def}
We briefly recall the relevant definitions for sheaf theory.
\begin{itemize}
    \item A \textbf{collection of families} $j$ on a category $\site$ consists of a set $j(U)$ for each $U \in \site$, whose elements $\{r_i: U_i \to U \} \in j(U)$ are families of morphisms over $U$. We call a collection of families $j$ on $\site$ a \textbf{coverage} if it satisfies the following property: for every $\{ r_i: U_i \to U \} \in j(U)$, and every map $g: V \to U$ in $\site$, then there exists a family $\{ t_j: V_j \to V \} \in j(V)$ such that $g t_j$ factors through some $r_i$. Namely for every $t_j$ there exists some $i$ and some map $s_j: V_j \to U_i$ making the following diagram commute:
\begin{equation} \label{eqn coverage def}
    \begin{tikzcd}
	{V_j} & {U_i} \\
	V & U
	\arrow["{t_j}"', from=1-1, to=2-1]
	\arrow["{s_j}", from=1-1, to=1-2]
	\arrow["{r_i}", from=1-2, to=2-2]
	\arrow["g"', from=2-1, to=2-2]
\end{tikzcd}
\end{equation}
The families $\{ r_i: U_i \to U \} \in j(U)$ are called \textbf{covering families} over $U$. If a map $r_i: U_i \to U$ belongs to a covering family $r \in j(U)$, then we say that $r_i$ is a \textbf{covering map}. If $\site$ is a category, and $j$ is a coverage on $\site$, then we call the pair $(\site, j)$ a \textbf{site}.
\item A \textbf{presheaf} on a category $\site$ is a functor $F: \site^{op} \to \ncat{Set}$. A morphism of presheaves is a natural transformation. An element $x \in F(U)$ for an object $U \in \site$ is called a \textbf{section} over $U$. If $f: U \to V$ is a map in $\site$, and $x \in F(V)$,  then we sometimes denote $F(f)(x)$ by $x|_U$. Let $\ncat{Pre}(\cat{C})$ denote the category of presheaves on $\cat{C}$.
\item  If $\{r_i: U_i \to U \}_{i \in I}$ is a covering family, then a \textbf{matching family} is a collection $\{x_i \}_{i \in I}$, $x_i \in F(U_i)$, such that given a diagram in $\site$ of the form
\begin{equation*}
\begin{tikzcd}
	V & {U_j} \\
	{U_i} & U
	\arrow["{r_i}"', from=2-1, to=2-2]
	\arrow["{r_j}", from=1-2, to=2-2]
	\arrow["f"', from=1-1, to=2-1]
	\arrow["g", from=1-1, to=1-2]
\end{tikzcd}
\end{equation*}
then $F(f)(x_i) = F(g)(x_j)$ for all $i,j \in I$. An \textbf{amalgamation} $x$ for a matching family $\{ x_i \}$ is a section $x \in F(U)$ such that $x_i|_U = x$ for all $i$.
\item Given a family of morphisms $r = \{ r_i: U_i \to U \}$ in a category $\site$, we say that a presheaf $F: \site^{op} \to \ncat{Set}$ is a \textbf{sheaf on $r$} if every matching family $\{ s_i \}$ of $F$ over $r$ has a unique amalgamation. If $j$ is a coverage on a category $\site$, we call $F$ a \textbf{sheaf} on $(\site, j)$ if it is a sheaf on every covering family of $j$. Let $\ncat{Sh}(\site)$ denote the full subcategory of $\ncat{Pre}(\site)$ whose objects are sheaves on $(\site, j)$.
\end{itemize}
\end{Def}

One can put a site structure on $\cart$ using the coverage of good open covers, see \cite[Section 4]{minichiello2022diffeological}. We call sheaves on $\cart$ \textbf{smooth sheaves}. There are many interesting examples of smooth sheaves. Every cartesian space defines a representable sheaf $yU$. Every manifold $M$ defines a sheaf by $U \mapsto C^\infty(U, M)$. There are also $\Omega^n$ and $\Omega^n_{\cl}$ for every $n \geq 0$, the sheaves of differential $n$-forms and closed differential $n$-forms respectively. The category $\ncat{Sh}(\cart)$ of smooth sheaves is ``extremely nice", being a Grothendieck topos \cite{maclane2012sheaves}.

A sheaf $X$ on $\cart$ is \textbf{concrete} if $X(U)$ is a subset of the set functions $U \to X(*)$ where $*$ is the terminal object in $\cart$. The representable sheaves $yU$ and the sheaves induced by manifolds $M$ are concrete, but $\Omega^n$ and $\Omega^n_{\cl}$ are not. 

The full subcategory $\ncat{ConSh}(\cat{C}) \hookrightarrow \ncat{Sh}(\cat{C})$ of concrete sheaves on a concrete site forms a quasitopos, which while not being a Grothendieck topos, is still a very ``nice" category \cite[Theorem 52]{baez2009convenient}.

\begin{Th}[{\cite[Prop 24]{baez2009convenient}}] \label{th baez-hoffnung theorem}
Let $\cart$ denote the site of cartesian spaces with the coverage of good open covers. Then there is an equivalence of categories
\begin{equation}
    \ncat{Diff} \simeq \ncat{ConSh}(\cart),
\end{equation}
where $\ncat{ConSh}(\cart)$ denotes the category of concrete sheaves on $\cart$. 
\end{Th}

\begin{Cor} \label{cor diff is a quasitopos}
The category $\ncat{Diff}$ is a quasitopos. This implies that it is a complete, cocomplete and cartesian closed category.
\end{Cor}

We refer to Theorem \ref{th baez-hoffnung theorem} as the \textbf{Baez-Hoffnung Theorem}\footnote{Strictly speaking, the Baez-Hoffnung theorem gives an equivalence between the category of what we call classical diffeological spaces and the category of concrete sheaves on the site of open subsets of cartesian spaces with open covers, see \cite[Appendix A]{minichiello2022diffeological} for a proof that this is an equivalent formulation.}. It is the starting point of the interaction of sheaf theory and diffeology. Many aspects of the study of diffeological spaces can be restated using sheaf theory, for example a differential $n$-form $\omega$ on a diffeological space $X$ as defined in \cite[Article 6.28]{iglesias2013diffeology} is equivalently a morphism $X \to \Omega^n$ of sheaves.

In \cite{minichiello2022diffeological} we took advantage of the Baez-Hoffnung Theorem to embed the category of diffeological spaces into the category of simplicial presheaves on $\cart$. We will delve into this idea in the next section. Once inside the category of simplicial presheaves, we can then take advantage of many homotopical tools. This in effect provides a way of obtaining a very powerful and expressive homotopy theory for diffeological spaces that subsumes the usual homotopy theory for diffeological spaces as considered in \cite[Chapter 5]{iglesias2013diffeology}.

\section{Simplicial Presheaves} \label{section simplicial presheaves}
In this section we detail the model categorical notions we will need for the remainder of the paper. We assume the reader is comfortable with model categories and simplicial homotopy theory, and recommend the following standard sources \cite{hirschhorn2009model}, \cite{hovey2007model}, \cite{goerss2012simplicial}, \cite{goerss2006model} for good references on the topics. See \cite[Section 5]{minichiello2022diffeological} for more details.

\begin{Def}
Let $\spre(\cart)$ denote the category whose objects are functors $\cart^{op} \to \ncat{sSet}$, which we call \textbf{simplicial presheaves}, and whose morphisms are natural transformations.
\end{Def}

Note that $\spre(\cart)$ is complete and cocomplete, with limits and colimits computed objectwise. There are two pairs of adjoint triples that give structure to $\spre(\cart)$.

\begin{equation} \label{eq adjoint triples for simplicial presheaves}
\begin{tikzcd}
	{\spre(\cart)} && {\ncat{sSet}}
	\arrow[""{name=0, anchor=center, inner sep=0}, "{(-)_c}"{description}, from=1-3, to=1-1]
	\arrow[""{name=1, anchor=center, inner sep=0}, "{\text{colim}_{\cart^{op}}}", curve={height=-30pt}, from=1-1, to=1-3]
	\arrow[""{name=2, anchor=center, inner sep=0}, "{\text{lim}_{\cart^{op}}}"', curve={height=30pt}, from=1-1, to=1-3]
	\arrow["\dashv"{anchor=center, rotate=-93}, draw=none, from=0, to=2]
	\arrow["\dashv"{anchor=center, rotate=-87}, draw=none, from=1, to=0]
\end{tikzcd}, \qquad
\begin{tikzcd}
	{\spre(\cart)} && {\pre(\cart)}
	\arrow[""{name=0, anchor=center, inner sep=0}, "{{}^c(-)}"{description}, from=1-3, to=1-1]
	\arrow[""{name=1, anchor=center, inner sep=0}, "{\pi_0}", curve={height=-30pt}, from=1-1, to=1-3]
	\arrow[""{name=2, anchor=center, inner sep=0}, "{(-)_0}"', curve={height=30pt}, from=1-1, to=1-3]
	\arrow["\dashv"{anchor=center, rotate=-89}, draw=none, from=1, to=0]
	\arrow["\dashv"{anchor=center, rotate=-91}, draw=none, from=0, to=2]
\end{tikzcd}
\end{equation}
where $(-)_c$ is the functor induced by restricting along the unique functor $\cart \to *$ and ${}^c(-)$ is the functor that sends a presheaf to the corresponding simplicial presheaf where all the simplicial face and degeneracy maps are the identity. We often don't use the notation ${}^c(-)$ explicitly, especially for representable presheaves, as it should be clear from context. The functors $\pi_0$ and $(-)_0$ are defined objectwise. For every $U \in \cart$, and simplicial presheaf $X$ on $\cart$, $\pi_0 X(U) = \pi_0 (X(U))$, the set of connected components of $X(U)$, and $(X)_0(U) = X(U)_0$, the set of vertices of $X(U)$.

\begin{Rem}
The above adjoint triples exist for any essentially small category $\cat{C}$ in place of $\cart$.
\end{Rem}

The category $\spre(\cart)$ is tensored, cotensored and enriched over $\ncat{sSet}$. Indeed, if $K$ is a simplicial set and $X$ is a simplicial presheaf, then
\begin{itemize}
    \item $X \otimes K$ is the simplicial presheaf defined objectwise by
    \begin{equation*}
        (X \otimes K)(U) = (X \times K_c)(U) = X(U) \times K.
    \end{equation*}
    \item $X^K$ is the simplicial presheaf defined objectwise by
    \begin{equation*}
        (X^K)(U) = X(U)^K,  
    \end{equation*}
    where for simplicial sets $K$ and $L$, $K^L$ denotes the simplicial function complex.
    \item for any two simplicial presheaves $X$ and $Y$, let $\u{\spre(\cart)}(X,Y)$ denote the simplicial set defined levelwise by
    \begin{equation*}
        \u{\spre(\cart)}(X,Y)_n = \spre(\cart)(X \otimes \Delta^n, Y).
    \end{equation*}
\end{itemize}
This structure is compatible in the sense of the following natural isomorphisms of simplicial sets
\begin{equation}
    \u{\spre(\cart)}(X \otimes K, Y) \cong \u{\spre(\cart)}(X,Y^K).
\end{equation}

The category $\spre(\cart)$ inherits several model structures from $\ncat{sSet}$. We say a map $f: X \to Y$ is a \textbf{projective weak equivalence} if it is an objectwise weak equivalence of simplicial sets, a \textbf{projective fibration} if it is an objectwise fibration, and a \textbf{projective cofibration} if it left lifts against all maps that are both projective weak equivalences and projective fibrations.

\begin{Th}[{\cite[Page 314]{bousfield1972homotopy}, \cite[Section A.2.6]{lurie2009higher}}]
The projective weak equivalences, fibrations and cofibrations define a proper, combinatorial, simplicial model category structure on $\spre(\cart)$, called the \textbf{projective model structure} on simplicial presheaves.
\end{Th}

Let $\H$ denote the category of simplicial presheaves equipped with the projective model structure. Note that \cite[Corollary 9.4]{dugger2001universal} describes a sufficient condition on simplicial presheaves to be projective cofibrant, and it implies that all representable presheaves, denoted $yU$ for a cartesian space $U$, are projective cofibrant.

Given a cartesian space $U$ and a good cover $\mathcal{U} = \{U_i \subseteq U \}$ of $U$, we can form the simplicial presheaf $\check{C}(\mathcal{U})$ defined levelwise by
\begin{equation*}
\check{C}(\mathcal{U})_n = \coprod_{i_0 \dots i_n} y(U_{i_0} \cap \dots \cap U_{i_{n}}). 
\end{equation*}
We call $\check{C}(\mathcal{U})$ the \textbf{\v{C}ech nerve} of $\mathcal{U}$. There is a canonical map $\pi: \check{C}(\mathcal{U}) \to yU$. Let $\check{C}$ denote the class of morphisms $\pi: \check{C}(\mathcal{U}) \to yU$ where $U$ ranges over the cartesian spaces and $\mathcal{U}$ ranges over the good open covers for $U$. 

\begin{Th}[{\cite[Theorem A.6]{dugger2004hypercovers}}]
The left Bousfield localization of $\H$ at $\check{C}$ exists. We call the resulting model structure the \textbf{\v{C}ech model structure} on $\spre(\cart)$, and denote it by $\cH$. It is similarly a proper, combinatorial and simplicial model category.\footnote{It is important to note that the projective/objectwise weak equivalences between simplicial presheaves are still weak equivalences in the \v{C}ech model structure. Furthermore, all \v{C}ech weak equivalences between $\infty$-stacks are objectwise weak equivalences.}
\end{Th}

The fibrant objects in $\cH$ are called \textbf{$\infty$-stacks}. They are those projective fibrant simplicial presheaves $X$ such that the canonical map
\begin{equation}
   \u{\spre(\cart)}(yU, X) \to \u{\spre(\cart)}(\check{C}(\mathcal{U}), X), 
\end{equation}
is a weak equivalence of simplicial sets, for every cartesian space $U$ and good cover $\mathcal{U}$ of $U$. Every sheaf and classical stack of groupoids on $\cart$, thought of as simplicial presheaves on $\cart$, is an $\infty$-stack. See \cite[Section 4.1]{minichiello2022diffeological} for more details.

The identity functors define a Quillen adjunction between the projective and \v{C}ech model structure on simplicial presheaves.
\begin{equation}
    \begin{tikzcd}
	{\H} && {\cH}
	\arrow[""{name=0, anchor=center, inner sep=0}, "{1_{\spre(\cart)}}", curve={height=-12pt}, from=1-1, to=1-3]
	\arrow[""{name=1, anchor=center, inner sep=0}, "{1_{\spre(\cart)}}", curve={height=-12pt}, from=1-3, to=1-1]
	\arrow["\dashv"{anchor=center, rotate=-90}, draw=none, from=0, to=1]
\end{tikzcd}
\end{equation}
Crucially, the (easy to compute) finite homotopy limits in $\H$ are preserved as homotopy limits in $\cH$, thanks to the following result.

\begin{Prop}[{\cite[Proposition 11.2]{rezktoposes2010}}] \label{prop model topos localization preserves finite homotopy limits}
The left Quillen functor $1_{\spre(\cart)} : \H \to \cH$ preserves finite homotopy limits.
\end{Prop}

Given simplicial presheaves $X$ and $Y$ on $\cart$, let $Q$ and $R$ denote cofibrant and fibrant replacement functors for $\cH$ respectively, then let
\begin{equation}
   \R\cH(X, Y) = \u{\spre(\cart)}(QX,RY).  
\end{equation}
We call $\R\cH(X,Y)$ the \textbf{derived mapping space} of $X$ and $Y$. If $X$ is already cofibrant, then we can take $Q = 1_{\spre(\cart)}$ and if $Y$ is already fibrant, we can take $R = 1_{\spre(\cart)}$.

If $X$ and $A$ are simplicial presheaves, then let
\begin{equation*}
    \check{H}^0_\infty(X,A) = \pi_0 \R \cH(X,A).
\end{equation*}
We call this the $0$th $\infty$-stack cohomology of $X$ with values in $A$.

If $A$ is a simplicial presheaf that is objectwise a simplicial group, then we let
\begin{equation*}
    \check{H}^1_\infty(X,A) = \pi_0 \R \cH(X, \overline{W}A),
\end{equation*}
where $\overline{W}$ is the delooping functor, see \cite[Definition 4.29]{minichiello2022diffeological}.

If $A$ is a simplicial presheaf such that $\overline{W}^k A$, which we call its $k$-fold delooping, exists for $k \geq 1$, then we say that $A$ is $k$-deloopable, and we let
\begin{equation*}
 \check{H}^k_\infty(X,A) = \pi_0 \R \cH(X, \overline{W}^k A). 
\end{equation*}
We call this the $k$th $\infty$-stack cohomology of $X$ with coefficients in $A$. If the $k$-fold delooping of a simplicial presheaf $A$ exists and is an $\infty$-stack, then we denote it by $\B^k A \coloneqq \overline{W}^k A$. The following result is well known, see \cite[Lemma 4.34]{minichiello2022diffeological} for a proof.

\begin{Lemma}
If $A$ is a presheaf of simplicial abelian groups, then $A$ is $k$-deloopable for all $k \geq 1$. 
\end{Lemma}

There is a convenient cofibrant replacement functor for $\cH$, given (in the notation of \cite[Section 4.2]{riehl2014categorical}) for a simplicial presheaf $X$ by the bar construction $B(X, \cart, y)$, where $y$ denotes the Yoneda embedding $y : \cart \hookrightarrow \spre(\cart)$. 

If $X$ is a diffeological space, then by the Baez-Hoffnung Theorem (Theorem \ref{th baez-hoffnung theorem}), we can consider it as a sheaf on $\cart$. Then ${}^c X$ is a simplicial presheaf\footnote{We will often not use the notation ${}^c X$ for diffeological spaces in what follows, as it should be apparent from context what category we are considering $X$ in.}. If we apply $Q$ to ${}^c X$ then this formula reduces to the simplicial presheaf given levelwise by
\begin{equation}
    QX_n = \coprod_{(f_{n-1}, \dots, f_0) \in N(\ncat{Plot}(X))_n} yU_{p_n} \otimes \Delta^n.
\end{equation}
See \cite[Section 4.2]{minichiello2022diffeological} for more details.

Using the above cofibrant replacement functor for a diffeological space $X$, if $A$ is a $\infty$-stack that is also objectwise a simplicial abelian group, then we can obtain an explicit description of its $k$th $\infty$-stack cohomology with values in $A$, given by the $k$th cohomology of the cochain complex obtained by taking the dual Dold-Kan correspondence functor applied to the cosimplicial abelian group 
\begin{equation}
\begin{tikzcd}
	{A(QX_0)} & {A(QX_1)} & {A(QX_2)} & \dots
	\arrow[shift left=2, from=1-1, to=1-2]
	\arrow[shift right=2, from=1-1, to=1-2]
	\arrow[shift left=2, from=1-2, to=1-3]
	\arrow[shift right=2, from=1-2, to=1-3]
	\arrow[from=1-2, to=1-3]
\end{tikzcd},
\end{equation}
the case where $A$ is an abelian diffeological group is given by \cite[Corollary 4.38]{minichiello2022diffeological}.

\begin{Ex} \label{ex diffeological bundles via stacks}
Let $G$ be a diffeological group, and consider the (strict) functor $\B G: \cart^{op} \to \ncat{Gpd}$ that sends a cartesian space $U$ to the groupoid\footnote{Here we use the convention discussed in \cite[Example 5.13]{minichiello2022diffeological} for $\B G$.}
\begin{equation}
    [C^\infty(U, G) \rightrightarrows *]
\end{equation}
Postcomposing with the nerve functor gives us a simplicial presheaf $N \B G$, which we will often just denote by $\B G$. By \cite[Theorem 5.17]{minichiello2022diffeological}, (referencing \cite[Lemma 3.3.29]{sati2022equivariant} and \cite[Proposition 4.13]{pavlov2022numerable}), $\B G$ is an $\infty$-stack. 

This $\infty$-stack takes a central role in the theory of diffeological principal $G$-bundles. For every cartesian space $U$, there is a canonical map of groupoids
\begin{equation}
    \B G(U) \to \ncat{DiffPrin}_G(U),
\end{equation}
that sends the point to the trivial diffeological principal $G$-bundle, and sends a map to $G$ to the corresponding automorphism of the trivial bundle. This map is an equivalence of groupoids.

Furthermore, if $X$ is a diffeological space and $QX$ is its cofibrant replacement, then $G$-cocycles on $X$ are equivalent to maps of $\infty$-stacks $QX \to \B G$ in the sense of \cite[Lemma 6.7]{minichiello2022diffeological}. In other words, for every diffeological space $X$, there is a weak equivalence
\begin{equation}
    \R \cH(X, \B G) \simeq N \ncat{DiffPrin}_G(X).
\end{equation}
Thus we say that $\B G$ classifies diffeological principal $G$-bundles. This implies that
\begin{equation}
    \check{H}^1_\infty(X, G) \cong \pi_0 \ncat{DiffPrin}_G(X),
\end{equation}
where $\pi_0 \ncat{DiffPrin}_G(X)$ denotes the set of isomorphism classes of diffeological principal $G$-bundles on $X$.
\end{Ex}

Let us now examine the left hand side of (\ref{eq adjoint triples for simplicial presheaves}). If $K$ is a simplicial set, then $K_c$ is the constant simplicial presheaf on $K$, namely $K_c(U) = K$ for all $U \in \cart$. The functors making up (\ref{eq adjoint triples for simplicial presheaves}) are important enough to warrant renaming. Notice that since $\R^0 = *$ is the terminal object in $\cart$, it is the initial object in $\cart^{op}$, thus $\lim_{U \in \cart^{op}} X(U) \cong X(*)$. For $K \in \ncat{sSet}$ and $X \in \spre(\cart)$, we set
\begin{equation}
    \Disc(K) = K_c, \qquad \Gamma(X) = \lim_{U \in \cart^{op}} X(U) \cong X(*), \qquad \Pi_\infty(X) = \underset{U \in \cart^{op}}{\colim} X(U).
\end{equation}
It turns out that $\Gamma$ has a further right adjoint, $\text{CoDisc} :  \ncat{sSet} \to \spre(\cart)$ defined objectwise by
\begin{equation*}
    \text{CoDisc}(K)(U) = K^{\Gamma(yU)}.
\end{equation*}
We say that $\Disc(K)$ is the \textbf{discrete simplicial presheaf} on $K$, $\Gamma(X)$ is the \textbf{global sections} of $X$, $\Pi_\infty(X)$ is the \textbf{fundamental $\infty$-groupoid} or \textbf{shape} of $X$, and that $\text{CoDisc}(K)$ is the \textbf{codiscrete simplicial presheaf} on $K$. In fact, all of these adjunctions are simplicially enriched adjunctions.

Thus we obtain the following triple of simplicially enriched adjunctions
\begin{equation} \label{eq cohesive adjunctions}
\begin{tikzcd}
	{\spre(\cart)} && {\ncat{sSet}}
	\arrow[""{name=0, anchor=center, inner sep=0}, "\Pi_\infty", curve={height=-30pt}, from=1-1, to=1-3]
	\arrow[""{name=1, anchor=center, inner sep=0}, "{\text{CoDisc}}", curve={height=-30pt}, from=1-3, to=1-1]
	\arrow[""{name=2, anchor=center, inner sep=0}, "{\text{Disc}}"{description}, curve={height=12pt}, from=1-3, to=1-1]
	\arrow[""{name=3, anchor=center, inner sep=0}, "\Gamma"{description}, curve={height=12pt}, from=1-1, to=1-3]
	\arrow["\dashv"{anchor=center, rotate=-92}, draw=none, from=3, to=1]
	\arrow["\dashv"{anchor=center, rotate=-90}, draw=none, from=2, to=3]
	\arrow["\dashv"{anchor=center, rotate=-88}, draw=none, from=0, to=2]
\end{tikzcd}
\end{equation}

\begin{Prop}[{\cite[Prop 4.1.30 and 4.1.32]{schreiber2013differential}}]
Each adjunction in (\ref{eq cohesive adjunctions}) is a simplicial Quillen adjunction, where $\spre(\cart)$ is given the \v{C}ech model structure $\cH$, and $\ncat{sSet}$ is given the Kan-Quillen model structure.
\end{Prop}

In \cite{schreiber2013differential}, Schreiber defines the three following endofunctors on the category of simplicial presheaves on $\cart$:
\begin{equation}
\begin{aligned}
\textstyle{\int} & = \text{Disc} \circ \Pi_\infty \\
\flat & = \text{Disc} \circ \Gamma \\
\sharp & = \text{CoDisc} \circ \Gamma
\end{aligned}
\end{equation}
called \textbf{shape}, \textbf{flat} and \textbf{sharp} respectively.

\begin{Rem}
Using the name shape functor for both $\Pi_\infty$ and $\int$ is justified by remembering that $\Disc : \ncat{sSet} \to \cH$ is fully faithful.
\end{Rem}

They give another pair of simplicial Quillen adjunctions
\begin{equation}
    \begin{tikzcd}
	{\cH} && {\cH}
	\arrow[""{name=0, anchor=center, inner sep=0}, "{\textstyle{\smallint}}"', curve={height=30pt}, from=1-3, to=1-1]
	\arrow[""{name=1, anchor=center, inner sep=0}, "\sharp", curve={height=-30pt}, from=1-3, to=1-1]
	\arrow[""{name=2, anchor=center, inner sep=0}, "\flat"{description}, from=1-1, to=1-3]
	\arrow["\dashv"{anchor=center, rotate=-90}, draw=none, from=0, to=2]
	\arrow["\dashv"{anchor=center, rotate=-90}, draw=none, from=2, to=1]
\end{tikzcd}
\end{equation}

Let us focus further on the shape functor $\Pi_\infty$. Let $\Delta^k_a$ denote the cartesian space defined by
\begin{equation}
    \Delta^k_a = \left \{ (x_0, \dots, x_k) \in \R^{k+1} \, : \, \sum_{i = 0}^k x_i = 1 \right \}.
\end{equation}
We call these \textbf{affine simplices}.

Let $\Sing_\infty : \spre(\cart) \to \ncat{sSet}$ be the functor defined objectwise by
\begin{equation}
    \Sing_\infty(X) = \text{hocolim}_{\ncat{\Delta}^{op}} \left( \begin{tikzcd}
	{X(\Delta^0_a)} & {X(\Delta^1_a)} & {X(\Delta^2_a)} & \dots
	\arrow[shift right=2, from=1-2, to=1-1]
	\arrow[shift left=2, from=1-2, to=1-1]
	\arrow[shift right=2, from=1-3, to=1-2]
	\arrow[shift left=2, from=1-3, to=1-2]
	\arrow[from=1-3, to=1-2]
	\arrow[shift right=3, from=1-4, to=1-3]
	\arrow[shift left=3, from=1-4, to=1-3]
	\arrow[shift right=1, from=1-4, to=1-3]
	\arrow[shift left=1, from=1-4, to=1-3]
\end{tikzcd}\right),
\end{equation}
where if we wish to be concrete, we can use the model of the homotopy colimit given by taking the diagonal of the above bisimplicial set. We call this the \textbf{smooth singular complex functor}.

\begin{Lemma} \label{lem sing preserves objectwise weak equiv}
The functor $\text{Sing}_\infty$ sends objectwise weak equivalences of simplicial presheaves to weak equivalences.
\end{Lemma}

\begin{proof}
This follows from \cite[Proposition 1.9]{goerss2012simplicial} and taking the diagonal to model the homotopy colimit.
\end{proof}

\begin{Prop} \label{prop two models for shape}
There are natural weak equivalences between the functors
\begin{equation}
    \Sing_\infty \simeq \Sing_\infty \circ Q \simeq \Pi_\infty \circ Q
\end{equation}
where $Q$ denotes a cofibrant replacement functor for $\cH$.
\end{Prop}

\begin{proof}
We mirror the proof given in \cite[Remark 4.12]{bunk2022r}, and since the left hand weak equivalence is shown there, we only prove the middle weak equivalence. If we restrict the smooth singular complex functor along the Yoneda embedding $\cart \hookrightarrow \spre(\cart)$, then we obtain a functor $\Sing_\infty : \cart \to \ncat{sSet}$, and there is a natural weak equivalence of functors $\Sing_\infty \xrightarrow{\sim} *$, where $*: \cart \to \ncat{sSet}$ is the constant functor on a point $* = \Delta^0$, this follows from \cite[Proposition 3.11]{bunk2022r}. Since all simplicial sets are cofibrant in the Quillen model structure on $\ncat{sSet}$, \cite[Corollary 5.2.5]{riehl2014categorical} implies that this induces a natural weak equivalence
\begin{equation}
    B(X, \cart, \Sing_\infty) \xrightarrow{\sim} B(X, \cart, *)
\end{equation}
of simplicial presheaves for every $X \in \spre(\cart)$. But $B(X, \cart, *) \cong \colim_{\cart^{op}} B(X, \cart, y) \cong \colim_{\cart^{op}} QX \cong \Pi_\infty QX$. This follows from the fact that $B(X, \cart, y) \cong QX$, which is just repackaging the definition of $QX$. Since $\Sing_\infty$ is a left adjoint, we have 
$$B(X, \cart, \Sing_\infty) \cong \Sing_\infty B(X, \cart, y) \cong \Sing_\infty QX.$$ 
This gives the second natural weak equivalence above.
\end{proof}

By Proposition \ref{prop two models for shape}, we will often refer to $\Sing_\infty(X)$ as the shape of $X$ as well. The shape functor has many wonderful properties. While we will not need all of the following results on the shape functor for this paper, we provide a concise listing of them here as such results are scattered throughout the literature.

\begin{Rem}
Since $\int$ is just applying the shape functor and then treating the resulting simplicial set as a constant simplicial presheaf, along with the fact that $\int \simeq \Disc \circ \Sing_\infty$, we will blur the distinction between $\Pi_\infty, \Sing_\infty$ and $\int$. We will use $\int$ when we wish to be ambiguous about which particular model of the shape functor we wish to use.    
\end{Rem}

\begin{Th}[{\cite[Theorem 4.15]{bunk2021sheaves}}] \label{th manifolds shape is underlying homotopy type}
Let $M$ be a finite dimensional smooth manifold, thought of as a simplicial presheaf on $\cart$, then
\begin{equation}
    \int M \simeq \Sing(M^{top})
\end{equation}
where $M^{top}$ is the underlying topological space of $M$ and $\Sing : \ncat{Top} \to \ncat{sSet}$ is the classical singular complex. In other words, the shape of $M$ is its underlying homotopy type.
\end{Th}

\begin{Rem}
It should be noted that Theorem \ref{th manifolds shape is underlying homotopy type} is really a consequence of a classical result known as the nerve theorem \cite{borsuk1948imbedding}.
\end{Rem}

\begin{Prop}[{\cite[Example 14.1]{pavlov2022projective}}]
Let $\Omega^n_{\cl}$ denote the sheaf on $\cart$ of closed differential $n$-forms. Its shape is
\begin{equation}
    \int \Omega^n_{\cl} \simeq \B^n \R^\delta.
\end{equation}
\end{Prop}

\begin{Lemma} \label{lem shape of discrete is identity}
Let $K$ be a simplicial set, then the counit
\begin{equation}
    \varepsilon_K: (\Pi_\infty \circ \Disc)(K) \to K.
\end{equation}
is an isomorphism. In other words, for discrete simplicial presheaves $K_c$, we have
\begin{equation*}
    \int K_c \cong K.
\end{equation*}
\end{Lemma}

\begin{proof}
This follows from the fact that $\Disc$ is fully faithful, and the unit of an adjunction where the right adjoint is fully faithful is an isomorphism.
\end{proof}

\begin{Prop}[{\cite[Theorem 3.4]{carchedi2015homotopy}}]
Let $M$ denote a simplicial manifold. Let $\u{M}$ denote the simplicial presheaf on $\cart$ defined degreewise by
\begin{equation*}
    \u{M}(U)_k = C^\infty(U, M_k).
\end{equation*}
Then the shape of $\u{M}$ is
\begin{equation}
    \int \u{M} \simeq || M||
\end{equation}
where $||M||$ denotes the homotopy type of the ``fat" geometric realization of $M$, see \cite[Section 3.2]{carchedi2015homotopy}.
\end{Prop}

\begin{Rem}
The above result is especially interesting when $M$ is the nerve of a Lie groupoid, as this says that the shape of a Lie groupoid (thought of as a simplicial presheaf) is weak equivalent to the homotopy type of the Lie groupoid's classifying space \cite[Section 2.2]{carchedi2015homotopy}.
\end{Rem}

\begin{Rem}
We should also mention that the shape functor has been used to great effect in what is now called the Smooth Oka Principle. See the following references \cite{berwickevans2022classifying}, \cite{sati2022equivariant}, \cite{clough2023homotopy}, \cite[Section 10]{pavlov2022numerable}.
\end{Rem}

Let $\Sing_D : \ncat{Diff} \to \ncat{sSet}$ be the functor defined levelwise by
\begin{equation}
    \Sing_D(X)_n = \ncat{Diff}(\Delta^n_a, X).
\end{equation}
We call this the \textbf{diffeological singular complex}. Note that by using the diagonal as a model for the homotopy colimit, for a diffeological space $X$ (actually any sheaf of sets on $\cart$), we have
\begin{equation}
   \Sing_D(X) \cong \Sing_\infty(X). 
\end{equation}

\begin{Prop}
Let $X$ be a diffeological space, then 
\begin{equation}
    \Pi_\infty(QX) \simeq \Sing_D(X) \simeq N \ncat{Plot}(X)
\end{equation}
where $N\ncat{Plot}(X)$ is the nerve of the category of plots of $X$.
\end{Prop}

\begin{proof}
We will only prove the second weak equivalence, as the first holds by the previous discussion. It is shown in \cite[Proposition 3.6]{bunk2022r} that $\Sing_\infty$ has a left and right adjoint (though only the right adjoint forms a Quillen adjunction), therefore we have
\begin{equation}
    \begin{aligned}
        \Pi_\infty(QX) & = \Pi_\infty \left( \int^{n \in \ncat{\Delta}} \coprod_{N\ncat{Plot}(X)_n} yU_{p_n} \otimes \Delta^n \right) \\
        & \cong \int^{n \in \ncat{\Delta}} \coprod_{N\ncat{Plot}(X)_n} \Pi_\infty(yU_{p_n}) \times (\Pi_\infty \circ \Disc)(\Delta^n) \\
        & \simeq \int^{n \in \ncat{\Delta}} \coprod_{N\ncat{Plot}(X)_n} * \times \Delta^n \\
        & \cong N\ncat{Plot}(X).
    \end{aligned}
\end{equation}
where the weak equivalence is given by Theorem \ref{th manifolds shape is underlying homotopy type} and Lemma \ref{lem shape of discrete is identity}.
\end{proof}

\section{The Irrational Torus} \label{section irrational torus}
In this section, we will show that if $K \subset \R$ is a diffeologically discrete subgroup of the real numbers, then the infinity stack cohomology of the irrational torus $T_K = \R^n / K$ (for any $n \geq 1$) with values in $\R^\delta$, is isomorphic to the group cohomology of $K$ with values in $\R^\delta$. This was first proved by Iglesias-Zemmour \cite[Page 15]{iglesias2023vcech} with his own version of diffeological \v{C}ech cohomology, which we will refer to as PIZ cohomology. In \cite{minichiello2022diffeological} we found maps between $\infty$-stack cohomology and PIZ cohomology, and showed that one of these maps is a retract, but it is still an open question as to whether these two cohomologies are isomorphic. 

The motivation for this section is two-fold. One is to support the conjecture that $\infty$-stack cohomology is isomorphic to PIZ cohomology. We do this by showing that they agree on one of the most important class of examples of diffeological spaces, the irrational tori. The second motivation is to show the power of $\infty$-topos theory and in particular $\infty$-stack cohomology through the use of the shape operation. The proof of Theorem \ref{th infinity stack cohomology of irrational torus is group cohomology} is completely different than that in \cite{iglesias2023vcech}, and it is conceptually more straightforward.

\begin{Def}
Suppose that $K \subset \R^n$ is a subgroup, and furthermore, when $K$ is given the subset diffeology of $\R$, it coincides with the discrete diffeology, so that every plot is constant. We call its quotient $T_K = \R^n / K$ the $n$-dimensional \textbf{$K$-irrational torus}\footnote{The word irrational comes from the example where $n =1$, and $K = \Z + \alpha \Z$ with $\alpha$ an irrational number. This is the most studied example of an irrational torus in diffeology. Interestingly, $\Z + \alpha \Z \subset \R$ with the subset topology is dense in $\R$, hence not discrete, however the subset diffeology is discrete.}.
\end{Def} 

We define the quotient map $\pi: \R^n \to T_K$ of diffeological spaces. This map is a diffeological principal $K$-bundle \cite[Article 8.15]{iglesias2013diffeology}, where $K$ is given the discrete diffeology. By the discussion in Example \ref{ex diffeological bundles via stacks}, the bundle $\pi$ is classified by a map of $\infty$-stacks $g_K: T_K \to N \ncat{DiffPrin}_K$. Furthermore, by \cite[Corollary 6.12]{minichiello2022diffeological}, we obtain the following cube, where the front face and back face are homotopy pullback squares and the maps going from the back face to the front face are all objectwise weak equivalences
\begin{equation} 
\begin{tikzcd}
	{\widetilde{\R^n}} && {\mathbf{E} K} \\
	& {\R^n} && {*} \\
	{QT_K} && {\B K} \\
	& {T_K} && {N \ncat{DiffPrin}_K}
	\arrow[from=2-2, to=2-4]
	\arrow["\pi"'{pos=0.4}, from=2-2, to=4-2]
	\arrow["{{g_K}}"', from=4-2, to=4-4]
	\arrow[from=2-4, to=4-4]
	\arrow[from=1-1, to=3-1]
	\arrow[from=1-3, to=3-3]
	\arrow[from=1-1, to=1-3]
	\arrow[from=1-1, to=2-2]
	\arrow[from=3-1, to=4-2]
	\arrow[from=1-3, to=2-4]
	\arrow[from=3-3, to=4-4]
	\arrow[from=3-1, to=3-3]
\end{tikzcd}
\end{equation}

Now if we apply the shape functor to this cube, the back face remains a homotopy pullback by the following result.

\begin{Prop}[{\cite[Proposition 3.3.8]{sati2022equivariant}}] \label{prop homotopy pullback of shape}
Let $K$ be a simplicial set, and let $f: X \to \Disc(K)$ and $g : Y \to \Disc(K)$ be maps of simplicial presheaves on $\cart$. Then there is a \v{C}ech weak equivalence
\begin{equation}
    \left( \smallint X \right) \times^h_{\Disc(K)} \left( \smallint Y \right) \simeq \smallint \left( X \times^h_{\Disc(K)} Y \right)
\end{equation}
between the homotopy pullbacks of the maps $\int f$ and $\int g$ and the shape of the homotopy pullback of the maps $f$ and $g$.
\end{Prop}

\begin{Rem}
For a model-category theoretic proof of Proposition \ref{prop homotopy pullback of shape}, use the argument of \cite[Theorem 4.1.34]{schreiber2013differential}, and mirror the argument of \cite[Proposition 4.1.35]{schreiber2013differential} model categorically. 
\end{Rem}

Therefore the front face must also be a homotopy pullback square. Now by Lemma \ref{lem sing preserves objectwise weak equiv}, Lemma \ref{lem shape of discrete is identity}, and Theorem \ref{th manifolds shape is underlying homotopy type} it follows that
\begin{equation}
    \smallint \widetilde{\R^n} \simeq \smallint \R^n \simeq *, \qquad \smallint Q T_K \simeq \smallint T_K, \qquad \smallint \mathbf{E} K \simeq \smallint * \cong *, \qquad \B K \cong \smallint \B K \simeq \smallint \ncat{DiffPrin}_K.  
\end{equation}
From this we obtain the main result of this section.

\begin{Th} \label{th infinity stack cohomology of irrational torus is group cohomology}
There is an isomorphism
\begin{equation}
    \check{H}^n_\infty(T_K, \R^\delta) \cong H^n_\text{grp}(K, \R^\delta)
\end{equation}
of abelian groups, for every $n \geq 0$, where $\R^\delta$ denotes the discrete group of real numbers, and where $H^n_{\text{grp}}(K, \R^\delta)$ denotes the group cohomology of $K$ with coefficients in $\R^\delta$.
\end{Th}

\begin{proof}
First we note that since $\R^\delta$ is discrete, $\R^\delta \cong \flat \R^\delta$. Similarly, $\B^n \R^\delta \cong \flat \B^n \R^\delta$ for all $n \geq 0$. Thus 
\begin{equation}
    \R \cH (T_K, \B^n \R^\delta) \cong \R \cH(T_K, \flat \B^n \R^\delta) \cong \R \cH(\smallint T_K, \B^n \R^\delta).
\end{equation}
Now from Proposition \ref{prop two models for shape}, there is a weak equivalence $\Sing_D(T_K) \to \int T_K$ of simplicial sets, and by \cite[Proposition 4.30]{christensen2014homotopy}, $\Sing_D(T_K)$ is a Kan complex. Now $\int N \ncat{DiffPrin}_K$ is a Kan complex since it is a simplicial abelian group. Furthermore since $T_K$ is diffeologically connected \cite[Section II]{iglesias2023vcech}, $\Sing_D(T_K)$ is connected. Thus the composite map $\Sing_D(T_K) \to \int N\ncat{DiffPrin}_K$ is a map of connected Kan complexes whose homotopy fiber is contractible. Therefore by the long exact sequence of homotopy groups \cite[Theorem 3.8.12]{cisinski2019higher}, it is a weak equivalence, which implies that $\int g_K$ is a weak equivalence.
Therefore we have
\begin{equation}
    \R \cH(\smallint T_K, \B^n \R^\delta) \simeq \R \cH(\smallint N \ncat{DiffPrin}_K, \B^n \R^\delta)\simeq \R \cH(\B K, \B^n \R^\delta).
\end{equation}
Which implies that
\begin{equation}
     \check{H}^n_\infty(T_K, \R^\delta) = \pi_0 \R \cH(\smallint T_K, \B^n \R^\delta) \cong \pi_0 \R \cH(\B K, \B^n \R^\delta) = \check{H}^n_\infty (\B K, \R^\delta).
\end{equation}
However since $\B K$ and $\B^n \R^\delta$ are discrete, $\text{Disc}$ is fully faithful and $\B^n \R^\delta$ is a Kan complex, we have
\begin{equation*}
    \R \cH(\B K, \B^n \R^\delta) \simeq \u{\ncat{sSet}}(\B K, \B^n \R^\delta),
\end{equation*}
where $\u{\ncat{sSet}}(\B K, \B^n \R^\delta)$ is the usual simplicial set function complex. It is then well known (see \cite[Example 8.2.3]{weibel1995introduction} for instance) that
\begin{equation}
    \pi_0 \u{\ncat{sSet}}(\B K, \B^n \R^\delta) \cong H^n_{\text{grp}}(K, \R^\delta).
\end{equation}
This proves the theorem.
\end{proof}

\section{The Dold-Kan Correspondence} \label{section dold kan}

In this section, we discuss the Dold-Kan correspondence, which is central to Section \ref{section cech de Rham obstruction}.

\begin{Rem}
For the remainder of this paper, by a vector space, we mean a real vector space, not necessarily of finite dimension. By a chain complex we mean a non-negatively graded chain complex of vector spaces. Let $\ncat{Ch}$ denote the category of chain complexes.    
\end{Rem}

\begin{Def}
Let $\ncat{Vect}$ denote the category whose objects are vector spaces and whose morphisms are linear maps. Let $\ncat{sVect} \coloneqq \ncat{Vect}^{\mathsf{\Delta}^{op}}$ denote the category of simplicial vector spaces.
\end{Def}

\begin{Prop}[{\cite[Proposition 4.2 and Theorem 4.13]{goerss2006model}, \cite[Section 1]{jardine2003presheaves}}]
The category $\ncat{sVect}$ admits a proper, combinatorial, simplicial model category structure where a morphism $f: X \to Y$ is a 
\begin{enumerate}
    \item weak equivalence if it is a weak homotopy equivalence of the underlying simplicial sets,
    \item fibration if it is a Kan fibration of the underlying simplicial sets,
    \item cofibration if it is degreewise a monomorphsim.
\end{enumerate}
We call this the \textbf{Kan-Quillen} model structure\footnote{Note that every object in $\ncat{sVect}$ is fibrant and cofibrant.} on $\ncat{sVect}$.
\end{Prop}

\begin{Prop}[{\cite[Theorem 1.5]{goerss2006model}, \cite[Section 1]{jardine2003presheaves}}] \label{prop projective model structure on chain complexes}
The category $\ncat{Ch}$ admits a proper, combinatorial, simplicial model category structure where a morphism $f: C \to D$ is a
\begin{enumerate}
    \item weak equivalence if it is a quasi-isomorphism of chain complexes, and
    \item fibration if $f_k : C_k \to D_k$ is surjective in degrees $k \geq 1$.
    \item cofibration if it is degreewise a monomorphism\footnote{For the projective model structure on chain complexes of $R$-modules for a general commutative ring $R$, we require these to be degreewise monomorphisms with projective cokernel. Since all vector spaces are projective as $\R$-modules, this condition is always satisfied.}.
\end{enumerate}
We call this the \textbf{projective model structure}\footnote{Note that every object in $\ncat{Ch}$ is fibrant and cofibrant.} on chain complexes.
\end{Prop}

There is an adjoint pair of functors,
\begin{equation}
\begin{tikzcd}
	{\ncat{Ch}} && {\ncat{sVect}}
	\arrow[""{name=0, anchor=center, inner sep=0}, "{\ncat{DK}}", curve={height=-12pt}, from=1-1, to=1-3]
	\arrow[""{name=1, anchor=center, inner sep=0}, "N", curve={height=-12pt}, from=1-3, to=1-1]
	\arrow["\dashv"{anchor=center, rotate=-90}, draw=none, from=0, to=1]
\end{tikzcd}
\end{equation}
which by the Dold-Kan correspondence \cite[Theorem 4.1]{goerss2006model} form an adjoint equivalence\footnote{This result actually holds for chain complexes and simplicial objects taking values in any idempotent complete, additive category, see \cite[Theorem 1.2.3.7]{lurie2017higher}.}.

\begin{Lemma}[{\cite[Section 4.1]{schwede2003equivalences}}]
The adjunction $N \dashv \ncat{DK}$ is a Quillen adjunction (\cite[Definition 8.5.2]{hirschhorn2009model}) between the model category structures on $\ncat{Ch}$ and $\ncat{sVect}$. Furthermore, the functors form a Quillen equivalence (\cite[Definition 8.5.20]{hirschhorn2009model}). In fact, since $N$ and $\ncat{DK}$ form an adjoint equivalence, it follows that $\ncat{DK} \dashv N$ is also an adjoint equivalence. Furthermore $N$ and $\ncat{DK}$ are both left and right Quillen functors.
\end{Lemma}

If we consider the category of simplicial sets $\ncat{sSet}$ with its usual Kan-Quillen model structure \cite[Theorem 1.22]{goerss2006model}, there is a simplicial Quillen adjunction
\begin{equation}
    \begin{tikzcd}
	{\ncat{sSet}} && {\ncat{sVect}}
	\arrow[""{name=0, anchor=center, inner sep=0}, "{\R[-]}", curve={height=-12pt}, from=1-1, to=1-3]
	\arrow[""{name=1, anchor=center, inner sep=0}, "U", curve={height=-12pt}, from=1-3, to=1-1]
	\arrow["\dashv"{anchor=center, rotate=-90}, draw=none, from=0, to=1]
\end{tikzcd}
\end{equation}
where $U$ denotes the forgetful functor, and $\R[-]$ denotes the functor that sends a simplicial set $X$ to the free simplicial vector space $\R X$, defined degreewise by $\R X_n = \R(X_n)$, where $\R(X_n)$ is the free vector space on the set $X_n$. Thus we obtain a Quillen adjunction
\begin{equation}
    \begin{tikzcd}
	{\ncat{sSet}} && {\ncat{Ch}}
	\arrow[""{name=0, anchor=center, inner sep=0}, "{N\R[-]}", curve={height=-12pt}, from=1-1, to=1-3]
	\arrow[""{name=1, anchor=center, inner sep=0}, "{U \ncat{DK}}", curve={height=-12pt}, from=1-3, to=1-1]
	\arrow["\dashv"{anchor=center, rotate=-90}, draw=none, from=0, to=1]
\end{tikzcd}
\end{equation}
which is furthermore a simplicial Quillen adjunction\footnote{We will often omit the functor $U$ in our notation.}.

Note that the Dold-Kan correspondence also provides a simplicial enrichment of $\ncat{Ch}$. Indeed, suppose $C$ and $D$ are chain complexes, then let $\u{\ncat{Ch}}(C,D)$ denote the simplicial vector space defined degreewise by
\begin{equation}
    \u{\ncat{Ch}}(C,D)_k = \ncat{Ch}(N \R \Delta^k \otimes C, D).
\end{equation}
This makes the Dold-Kan correspondence an enriched adjoint equivalence. This is also the simplicial enrichment mentioned in Proposition \ref{prop projective model structure on chain complexes}.

This also supplies $\ncat{Ch}$ with tensoring and cotensoring over $\ncat{sSet}$. Namely if $K$ is a simplicial set and $C$ is a chain complex then $C \otimes K$ is the chain complex $C \otimes N \R K$, and $C^K = N\R \u{\ncat{Ch}}(N \R K, C)$.

Now the category of chain complexes $\ncat{Ch}$ is also enriched over itself. Indeed, if $C$ and $D$ are chain complexes, then let $\MapCh(C,D)$ denote the chain complex defined as follows. First let us define the unbounded chain complex $\MapCh^\Z(C,D)$ defined in degree $k \in \Z$ by
\begin{equation}
    \MapCh^\Z(C,D)_k = \prod_{i \geq 0} \ncat{Vect}(C_i, D_{i + k}),
\end{equation}
with $d: \MapCh^\Z(C,D)_k \to \MapCh^\Z(C,D)_{k-1}$ defined for a map $f$ by
\begin{equation*}
    df = d_D f - (-1)^k f d_C.
\end{equation*}
We call an element of degree $k$ in $\MapCh^\Z(C,D)$ a degree $k$ map from $C$ to $D$.

\begin{Def}
If $C$ is an unbounded ($\Z$-graded) chain complex, then let $\tau_{\geq 0} C$ denote the chain complex defined degreewise by $(\tau_{\geq 0} C)_k = C_k$ for $k > 0$, and $(\tau_{\geq 0} C)_0 = \Z_0 C$, the set of $0$-cycles of $C$, i.e. those $x \in C_0$ such that $d x = 0$, the differential on $\tau_{\geq 0} C$ is induced by the differential on $C$. We call $\tau_{\geq 0} C$ the \textbf{smart truncation} of $C$. 
\end{Def}

Now given chain complexes $C$ and $D$, let 
\begin{equation} \label{eq internal hom of chain complexes}
\MapCh(C,D) = \tau_{\geq 0} \MapCh^\Z(C,D),
\end{equation}
denote the smart truncation applied to $\MapCh^\Z(C,D)$. This means that $\MapCh(C,D)_k = \MapCh^\Z(C,D)$ for $k > 0$, and $\MapCh(C,D)_0 \cong \ncat{Ch}(C,D)$. We refer to $\MapCh(C,D)$ as the \textbf{mapping chain complex} between $C$ and $D$.

\begin{Lemma}[{\cite[Example 4.3.2]{opadotun2021simplicial}}] \label{lem chain complexes hom and DK}
Let $C$ and $D$ be chain complexes, then we have an isomorphism of simplicial vector spaces
\begin{equation}
    \ncat{DK} \, \MapCh(C,D) \cong \u{\ncat{Ch}}(C,D).
\end{equation}
Further, this provides an isomorphism
\begin{equation}
    \MapCh(C,D) \cong N \u{\ncat{Ch}}(C,D).
\end{equation}
\end{Lemma}

An explicit description for the path space of a chain complex $C$, equivalently the cotensoring $C^{\Delta^1}$, is given in Appendix \ref{section proof}.

\begin{Def}
Let $\cat{C}$ be a small category. Then let $\ncat{ChPre}(\cat{C})$ denote the category whose objects are functors $\cat{C}^{op} \to \ncat{Ch}$, and whose morphisms are natural transformations. We call such functors presheaves of chain complexes.
\end{Def}

\begin{Prop}[{\cite[Section 11.6]{hirschhorn2009model}}]
The category $\ncat{ChPre}(\cat{C})$ admits a proper, combinatorial, simplicial model category structure where a morphism $f: C \to D$ is a
\begin{enumerate}
    \item weak equivalence if it objectwise a weak equivalence in the projective model structure on chain complexes, and
    \item fibration if it is objectwise a fibration in the projective model structure on chain complexes.
\end{enumerate}
We refer to this as the \textbf{(global) projective model structure} on presheaves of chain complexes.
\end{Prop}

Thus we obtain a similar simplicial Quillen pair
\begin{equation} \label{eq presheaves Dold Kan}
    \begin{tikzcd}
	{\ncat{sPre}(\cat{C})} && {\ncat{ChPre}(\cat{C})}
	\arrow[""{name=0, anchor=center, inner sep=0}, "{N\R[-]}", curve={height=-12pt}, from=1-1, to=1-3]
	\arrow[""{name=1, anchor=center, inner sep=0}, "{U \ncat{DK}}", curve={height=-12pt}, from=1-3, to=1-1]
	\arrow["\dashv"{anchor=center, rotate=-90}, draw=none, from=0, to=1]
\end{tikzcd}
\end{equation}
where $\spre(\cat{C})$ is equipped with the projective model structure. In Appendix \ref{section proof}, we will use (\ref{eq presheaves Dold Kan}), along with Proposition \ref{prop model topos localization preserves finite homotopy limits}, to compute homotopy pullbacks in $\cH$.

\section{Examples of $\infty$-stacks} \label{section examples of infinity stacks}

In this section we detail the $\infty$-stacks involved in this paper, and examine the $\infty$-stack cohomology of a diffeological space with coefficients in some of these example $\infty$-stacks.

\begin{Ex}
Given a finite dimensional smooth manifold $M$, the functor $U \mapsto C^\infty(U, M)$ defines a sheaf on $\cart$, and therefore an $\infty$-stack. The same goes for diffeological spaces. Given a diffeological space $X$, let $X^\delta$ denote the diffeological space with the same underlying set, but equipped with the discrete diffeology. As sheaves we have $X^\delta \cong \Disc(X(*)) = \flat X$.
\end{Ex}

\begin{Ex}
The presheaf of differential $k$-forms $\Omega^k$ and the presheaf of closed differential $k$-forms $\Omega^k_{\text{cl}}$ are sheaves of vector spaces on $\cart$ for every $k \geq 0$. Thus they are $\infty$-stacks. The de Rham differential defines a map of $\infty$-stacks $d : \Omega^k \to \Omega^{k+1}$ for all $k \geq 0$. 

There is a canonical map 
\begin{equation}
    \mc(\R) : \R \to \Omega^1
\end{equation}
of $\infty$-stacks, defined by the Yoneda Lemma as follows. Notice that $\R \in \cart$, so by the Yoneda lemma, a map $\omega : y \R \to \Omega^1$ is equivalent to an element $\omega \in \Omega^1(\R)$. There is a canonical element of the set of $1$-forms on $\R$, called the \textbf{Maurer-Cartan form} of $\R$. For a general Lie group $G$, we let $\mc(G)$ denote its Maurer-Cartan form. If we label the coordinate of $\R$ by $t$, then the Maurer-Cartan form is simply given by $\mc(\R) = dt$. Thus for a cartesian space $U$, the function $\mc(\R)(U) : \R(U) \to \Omega^1(U)$ acts by taking a smooth map $f : U \to \R$ and pulling back the Maurer-Cartan form $f^* \mc(\R) \in \Omega^1(U)$. Note that this is the same thing as $df$. In other words as maps of $\infty$-stacks, we have $\mc(\R) = d$.
\end{Ex}

\begin{Ex} \label{ex stack BG}
Let $G$ be a diffeological group. As discussed in Example \ref{ex diffeological bundles via stacks} the presheaf of groupoids $\B G$ given by
\begin{equation*}
U \mapsto [C^\infty(U,G) \rightrightarrows *]
\end{equation*}
is a stack, and is objectwise weak equivalent to the stack of diffeological principal $G$-bundles. We abuse notation as in \cite[Example 5.13]{minichiello2022diffeological} and also let $\B G$ denote the corresponding simplicial presheaf, which is an $\infty$-stack. Given a diffeological space $X$, a $G$-cocycle on $X$ as in Definition \ref{def cocycle} is precisely a map $QX \to \B G$ of simplicial presheaves. Theorem \ref{th diff cocycle theorem} shows that the resulting groupoid of $G$-cocycles on $X$ is equivalent to the groupoid of diffeological principal $G$-bundles on $X$. Thus $\check{H}^1_\infty(X, G) \coloneqq \check{H}^0_\infty(X, \B G)$ is the set of isomorphism classes of $G$-cocycles, which is isomorphic to the set of isomorphism classes of diffeological principal $G$-bundles.
\end{Ex}

\begin{Ex} \label{ex diff bundle with connection}
If $G$ is a Lie group with Lie algebra $\mathfrak{g}$, then let $\Omega^1(-,\mathfrak{g})//G$ denote the presheaf of groupoids 
\begin{equation*}
    U \mapsto [\Omega^1(U, \mathfrak{g}) \times C^\infty(U,G) \underset{s}{\overset{t}{\rightrightarrows}} \Omega^1(U, \mathfrak{g})]
\end{equation*}
where $t(\omega, g) = \omega$ and $s(\omega, g) = \text{Ad}_g^{-1}(\omega) + g^* \mc(G)$. The nerve of this presheaf of groupoids is an $\infty$-stack \cite[Proposition 3.2.5]{fiorenza2011cech}\footnote{Notice that the definition above is precisely the opposite of the corresponding $\infty$-stack considered in \cite[Section 3]{fiorenza2011cech}. This is because of the convention we use in \cite[Example 5.13]{minichiello2022diffeological}. However this makes no difference on the theory as we show in Appendix \ref{section diff bundles with connection}.}. This is the $\infty$-stack that classifies principal $G$-bundles with connection. We will often abuse notation and write $\Omega^1(-,\mathfrak{g})//G$ to refer to the presheaf of groupoids and the simplicial presheaf obtained by taking the nerve construction.
Note that a map $QX \to \Omega^1(-, \mathfrak{g})//G$ is equivalent to the data of a $G$-cocycle $g_{f_0}: U_{p_1} \to G$ and a collection $\{ A_{p_0} \}_{p_0 \in \ncat{Plot}(X)}$ of $1$-forms $A_{p_0} \in \Omega^1(U_{p_0}, \mathfrak{g})$ such that for every map $f_0: U_{p_1} \to U_{p_0}$ of plots we have
\begin{equation}\label{eq classical diffeological connection equation}
    A_{p_1} = \text{Ad}_{g_{f_0}}^{-1}(f_0^* A_{p_0}) + g_{f_0}^* \mc(G).
\end{equation}
Let us call this collection of data $(g, A) = (\{g_{f_0} \}, \{A_{p_0}\})$ a \textbf{$G$-cocycle with connection}. We show that this definition of connection is equivalent to the one given in \cite[Definition 3.2.1]{waldorf2012transgression} in Appendix \ref{section diff bundles with connection}.
\end{Ex}

\begin{Rem}
The following examples of simplicial presheaves can be checked to be $\infty$-stacks by using \cite[Corollary 6.2]{pavlov2022numerable}. One simply needs to notice that the examples that follow are presheaves of bounded chain complexes, and can thus be thought of equivalently as presheaves of cochain complexes, and that the homotopy descent condition for presheaves of cochain complexes is equivalent to the condition of Dold-Kan applied to the presheaves of chain complexes to be $\infty$-stacks.
\end{Rem}

\begin{Ex} \label{ex R bundle k gerbes}
Given a sheaf $A$ of abelian groups on $\cart$, with $k \geq 1$, the simplicial presheaf $\B^k A$ is obtained by taking Dold-Kan of the presheaf of chain complexes $[A \to 0 \to \dots \to 0]$. When $A$ is an abelian diffeological group, then $\B A$ is the $\infty$-stack that classifies diffeological principal $A$-bundles, as shown in \cite{minichiello2022diffeological}. 

In this paper we will consider the example $\B^k \R$. Given a diffeological space $X$, a map $QX \to \B^k \R$ consists of a $g \in \R(QX_k)$ such that $\delta g = 0$, see Appendix \ref{section totalization}. We call these \textbf{$\R$-bundle $(k-1)$-gerbes}. Thus a diffeological principal $\R$-bundle is precisely a $\R$-bundle $0$-gerbe.

There is a vast literature on bundle gerbes in differential geometry such as \cite{murray1996bundle}, \cite{bunk2021gerbes}, \cite{stevenson2004bundle}. Typically bundle gerbes are defined as geometric objects, and then shown to define cohomology classes through cocycles such as above. However, giving descriptions of bundle $k$-gerbes as geometric objects becomes difficult and tedious as $k$ grows. Their description as cocycles is much more economical, and is all we need for this paper. There should be no difficulty in translating between the geometric description of diffeological bundle $1$-gerbes, such as in \cite{waldorf2012transgression} and the cocycle description we give here, but we leave this to future work.
\end{Ex}

\begin{Ex}
For $k \geq 1$, let $\B^k_\nabla \R$\footnote{Note that the $\B^k$ in $\B^k_\nabla \R$ is just notation, it is not actually the delooping of anything.} denote the simplicial presheaf obtained by applying Dold-Kan to the following presheaf of chain complexes\footnote{Note that the way this chain complex is written, $\Omega^k$ is in degree $0$.}
\begin{equation}
    [\R \xrightarrow{d} \Omega^1 \xrightarrow{d} \Omega^2 \to \dots \to \Omega^k].
\end{equation}
This complex is often referred to as the \textbf{Deligne complex} when $\R$ is replaced with $U(1)$ and $d : \R \to \Omega^1$ is replaced with $d \log : U(1) \to \Omega^1$. For $k \geq 2$, $\B^k_\nabla \R$ classifies $\R$-bundle $(k-1)$-gerbes with connection, and $\infty$-stack cohomology with values in $\B^k_\nabla \R$ is called \textbf{pure differential cohomology}\footnote{We also recommend \cite[Section 3.2]{Myers2021} for a discussion of how pure differential cohomology fits into the hexagon diagram of differential cohomology.} in \cite[Definition 3.2.10]{Myers2021}. Thus we call $\B^k_\nabla \R$ the \textbf{pure $k$-Deligne complex}. 

Given a diffeological space $X$, we will denote $\infty$-stack cohomology with values in the pure $k$-Deligne complex by
\begin{equation*}
    \check{H}^k_{\infty, \nabla}(X, \R) \coloneqq \check{H}^0_\infty(X, \B^k_\nabla \R).
\end{equation*}

Let us also note that when $k = 1$, we have an objectwise weak equivalence of $\infty$-stacks
\begin{equation*}
    \B_{\nabla} \R \simeq (\Omega^1(-)// \R)^{op}.
\end{equation*}
To see this, note that both of the simplicial presheaves are identical in simplicial degrees $0$ and $1$. This is because $\R$ is abelian, so $t(\omega, g) = \text{Ad}_g^{-1}(\omega) + g^* \mc(\R) = \omega + dg$ in $(\Omega^1(-)//\R)^{op}$, which is precisely the face map $d_0: \B_\nabla \R_1 \to \B_\nabla \R_0$. Since $(\Omega^1(-)//\R)^{op}$ is the nerve of a presheaf of groupoids, it is $2$-coskeletal, and therefore its $k$-homotopy groups are trivial for $k \geq 2$. The objectwise homotopy groups of $\B_{\nabla}\R$ are given by the objectwise homology of the chain complex by the Dold-Kan correspondence, and thus are also trivial for $k \geq 2$, thus they are objectwise weak equivalent. The distinction between $\Omega^1(-)// \R$ and $(\Omega^1(-)//\R)^{op}$ is because of \cite[Example 5.13]{minichiello2022diffeological}, so technically $\B_{\nabla} \R$ classifies diffeological principal $\R^{op}$-bundles with opposite connection, but this distinction is immaterial to the theory and we sweep it under the rug, and we say that the above classifies diffeological principal $\R$-bundles with connection.

If $X$ is a diffeological space, then as we will see in Example \ref{ex bundle gerbe with connection}, a $\R$-bundle $(k-1)$-gerbe with connection on $X$ is given by the data
\begin{equation}
    (\omega^k, \omega^{k-1}, \, \dots \,, \omega^1, g) \in \Omega^k(QX_0) \oplus \Omega^{k-1}(QX_1) \oplus \dots \oplus \Omega^1(QX_{k-1}) \oplus \R(QX_k),
\end{equation}
such $D(\omega^k, \dots, g) = 0$ in the double complex $\Omega^i(QX_j)$, see Appendix \ref{section totalization}. We will let $[\omega^k, \, \dots \, , g]$ denote the isomorphism class it represents in $\check{H}^k_{\infty, 
\nabla}(X, \R)$. 
\end{Ex}

\begin{Ex} \label{ex flat bundle gerbes} For $k \geq 1$, consider the $\infty$-stack $\B^k \R^\delta$. This $\infty$-stack classifies diffeological principal $\R^\delta$-bundles. However, note that there is a map of presheaves of chain complexes
\begin{equation} \label{eq different presentations of B R delta}
    [\R^\delta \to 0 \to \dots \to 0] \to [\R \xrightarrow{d} \Omega^1 \xrightarrow{d} \cdots \xrightarrow{d} \Omega^k_{\cl}]
\end{equation}
given by the inclusion $\R^\delta \hookrightarrow \R$. Furthermore, this map is an objectwise quasi-isomorphism, by the Poincare lemma. Thus we will take the right hand side of (\ref{eq different presentations of B R delta}) to be the model of $\B^k \R^\delta$ we will use for the rest of this paper. From this presentation, it is easy to see that $\B^k \R^\delta$ is equivalently the $\infty$-stack that classifies diffeological principal $\R$-bundle $(k-1)$-gerbes with flat connection.
\end{Ex}

\begin{Ex}
For $k \geq 1$, consider the $\infty$-stack $\B^k \Omega^1_{\cl}$. There is a map of presheaves of chain complexes
\begin{equation}
    [\Omega^1_{\cl} \to 0 \to \dots \to 0] \to [\Omega^1 \xrightarrow{d} \Omega^2 \xrightarrow{d} \dots \xrightarrow{d} \Omega^k_{\cl}]
\end{equation}
and this map is an objectwise quasi-isomorphism again by the Poincare lemma. We take the right hand side to be the model we will use for $\B^k \Omega^1_{\cl}$ for the rest of the paper.
\end{Ex}

\begin{Ex}
For $k \geq 1$, let $\Omega^{1 \leq \bullet \leq k}$ denote the simplicial presheaf obtained by applying Dold-Kan to the following presheaf of chain complexes
\begin{equation}
    [\Omega^1 \xrightarrow{d} \Omega^2 \to \dots \to \Omega^k].
\end{equation}
\end{Ex}

If $X$ is a diffeological space and $A$ is a $k$-deloopable $\infty$-stack, recall that its $\infty$-stack cohomology is given by
\begin{equation*}
   \check{H}^k_\infty(X, A) = \R \cH(X, \B^k A). 
\end{equation*}
Let us compute an example of $\infty$-stack cohomology for a diffeological space $X$ with values in the $\infty$-stack $\B \Omega^1$, as it will be emblematic of how we compute $\infty$-stack cohomology for all of the relevant examples presented in this section. In Appendix \ref{section totalization} we go into detail on how to compute such examples.

\begin{Ex} \label{ex B omega example}
Let $A' = \B \Omega^1 = [\Omega^1 \to 0]$ and $A = \DK \, A'$. Let us compute $H^0(X,A)$ for a diffeological space $X$ using Proposition \ref{prop cohomology with coefficients in a presheaf of chain complexes}. Consider the double complex $A'(QX)$,
\begin{equation}
    \begin{tikzcd}
	{\Omega^1(QX_0)} & {\Omega^1(QX_1)} & {\Omega^1(QX_2)} & \dots \\
	0 & 0 & 0 & \dots
	\arrow[from=1-1, to=2-1]
	\arrow["\delta", from=1-1, to=1-2]
	\arrow[from=2-1, to=2-2]
	\arrow[from=1-2, to=2-2]
	\arrow["-\delta", from=1-2, to=1-3]
	\arrow[from=2-2, to=2-3]
	\arrow[from=1-3, to=2-3]
        \arrow[from=1-3, to=1-4]
        \arrow[from=2-3, to=2-4]
\end{tikzcd}
\end{equation}
so that $\tot A'(QX)$ is the chain complex
\begin{equation}
    \tot A'(QX) = [ \Omega^1(QX_0) \xrightarrow{D} \ker( -\delta : \Omega^1(QX_1) \to \Omega^1(QX_2))]
\end{equation}
Thus a $0$-cycle in $\tot A'(QX)$ consists of a collection $\{ \omega_{f_0} \in \Omega^1(U_{p_1}) \}_{f_0 : U_{p_1} \to U_{p_0}}$ of $1$-forms for every map of plots of $X$ such that $-\delta \omega = 0$, which is equivalent to the condition that for every pair of composable plot maps $U_{p_2} \xrightarrow{f_1} U_{p_1} \xrightarrow{f_0} U_{p_0}$ we have
\begin{equation*}
    (\delta \omega)_{(f_1,f_0)} = f_1^* \omega_{f_0} - \omega_{f_0 f_1} + \omega_{f_1} = 0.
\end{equation*}
Two such $0$-cocycles are cohomologous $\omega \sim \omega'$ if there exists a collection $\{ \lambda_{p_0} \in \Omega^1(U_{p_0}) \}_{p_0 \in \ncat{Plot}(X)}$ of $1$-forms for every plot of $X$ such that 
\begin{equation*}
    (D\lambda)_{f_0} = (\delta \lambda)_{f_0} = \omega'_{f_0} - \omega_{f_0}
\end{equation*}
for every map $f_0 : U_{p_1} \to U_{p_0}$ of plots of $X$. Thus $\check{H}^0_\infty(X, \B \Omega^1) = \check{H}^1_\infty(X, \Omega^1)$ is precisely analogous to the term $H^{1,1}_\delta$ in \cite[Section 4.4]{iglesias2023vcech}.
\end{Ex}

\begin{Ex} \label{ex cohomology of bundle with connection}
Let us compute $\check{H}^1_{\infty, \nabla}(X, \R) \coloneqq \check{H}^0_\infty(X,\B_\nabla \R)$ for a diffeological space $X$. Consider the double complex
\begin{equation}
\begin{tikzcd}
	{\R(QX_0)} & {\R(QX_1)} & {\R(QX_2)} & \dots \\
	{\Omega^1(QX_0)} & {\Omega^1(QX_1)} & {\Omega^1(QX_2)} & \dots
	\arrow["d"', from=1-1, to=2-1]
	\arrow["\delta", from=1-1, to=1-2]
	\arrow["{-\delta}", from=1-2, to=1-3]
	\arrow["{-\delta}", from=2-1, to=2-2]
	\arrow["d"', from=1-2, to=2-2]
	\arrow["d"', from=1-3, to=2-3]
	\arrow["\delta", from=2-2, to=2-3]
	\arrow[from=1-3, to=1-4]
	\arrow[from=2-3, to=2-4]
\end{tikzcd}
\end{equation}
where $\R(QX_i) = C^\infty(QX_i, \R)$. A $0$-cocycle is the data of a map $g: QX_1 \to \R$ and a $1$-form $A \in \Omega^1(QX_0)$ such that $-\delta g = 0$ and $ -\delta A = d g$. The condition $-\delta g = \delta g = 0$ is equivalent to the condition $g_{f_0 f_1} = f_1^* g_{f_0} + g_{f_1}$ for every pair of composable plot maps $U_{p_2} \xrightarrow{f_1} U_{p_1} \xrightarrow{f_0} U_{p_0}$, which is precisely the condition for $g$ to be a cocycle defining a diffeological principal $\R$-bundle on $X$, see \cite[Section 5]{minichiello2022diffeological}. The condition $\delta A = dg$ is equivalent to the condition that for every map of plots $f_0 : U_{p_1} \to U_{p_0}$ we have $A_{p_1} - f_0^*A_{p_0} = dg_{f_0}$, which is precisely the equation for a connection on a diffeological principal $\R$-bundle, see Appendix \ref{section diff bundles with connection}. Given two $0$-cocycles $(A,g)$ and $(A', g')$, a $1$-coboundary consists of an element $h \in \R(QX_0)$ such that $\delta h = g' - g$ and $d h = A' - A$. This is precisely the definition of a morphism of $G$-cocycles with connection, see Definition \ref{def category of cocycles with connection}. Thus an element of $\check{H}^1_{\infty, \nabla}(X, \R)$ is an isomorphism class of a diffeological principal $\R$-bundle on $X$ with connection.
\end{Ex}

\begin{Ex} \label{ex bundle gerbe with connection}
Let us compute $\check{H}^2_{\infty, \nabla}(X, \R)$ for a diffeological space $X$. Consider the double complex
\begin{equation}
\begin{tikzcd}
	{\R(QX_0)} & {\R(QX_1)} & {\R(QX_2)} & \dots \\
	{\Omega^1(QX_0)} & {\Omega^1(QX_1)} & {\Omega^1(QX_2)} & \dots \\
	{\Omega^2(QX_0)} & {\Omega^2(QX_1)} & {\Omega^2(QX_2)} & \dots
	\arrow["d"', from=1-1, to=2-1]
	\arrow["{-\delta}", from=1-1, to=1-2]
	\arrow["\delta", from=1-2, to=1-3]
	\arrow["\delta", from=2-1, to=2-2]
	\arrow["d"', from=1-2, to=2-2]
	\arrow["d"', from=1-3, to=2-3]
	\arrow["{-\delta}", from=2-2, to=2-3]
	\arrow[from=1-3, to=1-4]
	\arrow[from=2-3, to=2-4]
	\arrow["d"', from=2-1, to=3-1]
	\arrow["{-\delta}", from=3-1, to=3-2]
	\arrow["d"', from=2-2, to=3-2]
	\arrow["\delta", from=3-2, to=3-3]
	\arrow[from=3-3, to=3-4]
	\arrow["d"', from=2-3, to=3-3]
\end{tikzcd}
\end{equation}
Then a $0$-cycle in $\tot \B^2_\nabla \R(QX)$ is an element $(\omega, A, g) \in \Omega^2(QX_0) \oplus \Omega^1(QX_1) \oplus \R(QX_2)$ such that $D(\omega, A, g) = 0$. This is equivalent to the equations $-\delta \omega = d A$, $-\delta A = d g$ and $-\delta g = 0$. The equation $-\delta g = \delta g = 0$ is equivalent to the condition that for every triple of composable maps of plots $(f_2, f_1, f_0)$ of $X$, we have 
\begin{equation}
    (\delta g)_{(f_2, f_1, f_0)} = f_2^* g_{(f_1, f_0)} - g_{(f_1 f_2, f_0)} + g_{(f_2, f_0 f_1)} - g_{(f_2, f_1)} = 0. 
\end{equation}
This is the diffeological analogue of the cocycle data of a $\R$-bundle gerbe on $X$. The other two equations $-\delta \omega = d A$ and $-\delta A = d g$ are the diffeological analogue of the cocycle data of a connection on a $\R$-bundle gerbe. Thus $\check{H}^2_{\infty, \nabla}(X, \R) = \check{H}^0_{\infty}(X, \B^2_\nabla \R)$ can be taken as the definition of the abelian group of isomorphism classes of diffeological $\R$-bundle $1$-gerbes with connection on $X$. The story for $k \geq 2$ is exactly analogous, and thus we take our definition of the abelian group of isomorphism classes of diffeological $\R$-bundle $(k-1)$-gerbes with connection to be $\check{H}^k_{\infty, \nabla}(X, \R) = \check{H}^0_{\infty}(X, \B^k_{\nabla} \R)$. 
\end{Ex}

\section{The \v{C}ech de Rham Obstruction} \label{section cech de Rham obstruction}

In this section, we obtain a diffeological \v{C}ech-de Rham obstruction exact sequence in every degree from a homotopy pullback diagram of $\infty$-stacks. In degree $1$, our exact sequence is analogous to \cite{iglesias2023vcech}.

Let $X$ be a diffeological space. In \cite{iglesias2023vcech}, Iglesias-Zemmour constructs the following exact sequence of vector spaces
\begin{equation} \label{eq patricks sequence}
    0 \to H^1_{dR}(X) \to \check{H}^1_{PIZ}(X,\R^\delta) \to {}^d E^{1,0}_2(X) \xrightarrow{c_1} H^2_{dR}(X) \to \check{H}^2_{PIZ}(X, \R^\delta)
\end{equation}
using the five term exact sequence coming from a diffeological version of the \v{C}ech-de Rham bicomplex spectral sequence. The vector space ${}^d E^{1,0}_2(X)$ is the subspace of isomorphism classes of diffeological principal $\R$-bundles on $X$ that admit a connection, and the vector spaces $\check{H}^k_{\text{PIZ}}(X, \R^\delta)$ are Iglesias-Zemmour's version of diffeological \v{C}ech cohomology, which we refer to as PIZ cohomology. The relationship between PIZ cohomology and $\infty$-stack cohomology is only partially understood, see \cite[Section 5]{minichiello2022diffeological}.

The exact sequence (\ref{eq patricks sequence}) demonstrates the obstruction to the \v{C}ech-de Rham Theorem holding for diffeological spaces. For finite dimensional smooth manifolds, all principal $\R$-bundles are trivial, as they have contractible fiber, and thus the obstruction vanishes. However, there are diffeological spaces (the irrational torus for example) that have nontrivial principal $\R$-bundles that admit connections \cite[Article 8.39]{iglesias2013diffeology}. 

We construct and geometrically interpret the obstruction to the \v{C}ech-de Rham isomorphism in each degree $k \geq 1$ via $\infty$-stacks. 

\begin{Th} \label{th hopullback diagram}
For every $k \geq 1$, there exists a commutative diagram of $\infty$-stacks of the following form
\begin{equation}
\begin{tikzcd}
	{*} & {\B^k \R^\delta} & {*} & {*} \\
	{*} & {\B^k_\nabla \R} & {\Omega^{k+1}_\cl} & {\Omega^{k+1}} \\
	& {\B^k \R} & {\B^k \Omega^1_{\cl}} & {\Omega^{1 \leq \bullet \leq k+1}} \\
	& {*} & {\B^{k+1} \R^\delta} & {\B^{k+1}_\nabla \R}
	\arrow[from=1-1, to=2-1]
	\arrow[from=1-1, to=1-2]
	\arrow[from=2-1, to=2-2]
	\arrow[from=1-2, to=2-2]
	\arrow[from=2-2, to=2-3]
	\arrow[from=1-2, to=1-3]
	\arrow[from=1-3, to=2-3]
	\arrow[from=2-2, to=3-2]
	\arrow[from=2-3, to=2-4]
	\arrow[from=1-4, to=2-4]
	\arrow[from=1-3, to=1-4]
	\arrow[from=2-4, to=3-4]
	\arrow[from=2-3, to=3-3]
	\arrow[from=3-2, to=3-3]
	\arrow[from=3-3, to=3-4]
	\arrow[from=4-3, to=4-4]
	\arrow[from=3-4, to=4-4]
	\arrow[from=3-2, to=4-2]
	\arrow[from=4-2, to=4-3]
	\arrow[from=3-3, to=4-3]
\end{tikzcd}
\end{equation}
furthermore every commutative square in this diagram is a homotopy pullback square in $\cH$.
\end{Th}

We prove Theorem \ref{th hopullback diagram} in Appendix \ref{section proof}. 

\begin{Cor}\label{cor prePIZ exact sequence}
For every diffeological space $X$, there is an exact sequence of vector spaces\footnote{Near the completion of this paper, we learned that an analogous exact sequence was also obtained in \cite[Page 27]{Myers2021} using completely different methods in the framework of homotopy type theory.}
\begin{equation} \label{eq degree k prePIZ exact sequence}
   0 \to \check{H}^k_{\infty}(X, \R^\delta) \to \check{H}^k_{\infty, \nabla}(X, \R) \to \Omega^{k+1}_{\cl}(X) \to \check{H}^{k+1}_{\infty}(X, \R^\delta).
\end{equation}
\end{Cor}

\begin{proof}
This follows from Theorem \ref{th hopullback diagram} and Lemma \ref{lem fiber sequence gives exact sequence}.
\end{proof}

Let us explore the consequences of Corollary \ref{cor prePIZ exact sequence} in the case where $X$ is the irrational torus. Let $K = \Z + \alpha \Z$ be the subgroup of $\R$ consisting of those $x \in \R$ of the form $n + \alpha m$ where $n, m$ are integers and $\alpha$ is an irrational number. Let $T_\alpha = \R / \Z + \alpha \Z$. We can fully compute the de Rham and $\infty$-stack cohomology of $T_\alpha$. Every differential form on $T_\alpha$ is closed \cite[Exercise 119]{iglesias2013diffeology} so $\Omega^k(T_\alpha) = \Omega^k_{\cl}(T_\alpha) = H^k_{\dR}(T_\alpha)$, and furthermore $\Omega^k_{\cl}(T_\alpha) \cong \Lambda^k \R$ by \cite[Exercise 105]{iglesias2013diffeology}. Therefore we have
\begin{equation} \label{eq derham of irrational torus}
   \Omega^k_{\cl}(T_\alpha) = H_{\dR}^k(T_\alpha) \cong \begin{cases}
       \R, & k = 0, 1 \\
       0, & k > 1.
   \end{cases} 
\end{equation}
Now by Theorem \ref{th infinity stack cohomology of irrational torus is group cohomology}, we have
\begin{equation} \label{eq cech of irrational torus}
    \check{H}^k_\infty(T_\alpha, \R^\delta) \cong \check{H}^0_\infty(\B (\Z + \alpha \Z), \B^k \R^\delta) \cong \check{H}^0_\infty(\B \Z^2, \B^k \R^\delta) \cong \check{H}^k_\infty(T^2, \R^\delta) \cong \begin{cases}
        \R, & k = 0, 2 \\
        \R^2, & k = 1 \\
        0, & k > 2,
    \end{cases} 
\end{equation}
where $T^2$ denotes the usual $2$-dimensional torus.

From Corollary \ref{cor prePIZ exact sequence}, setting $k = 1$, we obtain the exact sequence
\begin{equation}
    0 \to \check{H}^1_{\infty}(T_\alpha, \R^\delta) \to \check{H}^1_{\infty, \nabla}(T_\alpha, \R) \to \Omega^2_{\cl}(T_\alpha) \to \check{H}^2_{\infty}(T_\alpha, \R^\delta).
\end{equation}
Since $\Omega^2_{\cl}(T_\alpha) = 0$, this implies that $\check{H}^1_{\infty, \nabla}(T_\alpha, \R) \cong \check{H}^1_{\infty}(T_\alpha, \R^\delta) \cong \R^2$.

From Corollary \ref{cor prePIZ exact sequence}, setting $k = 2$, we obtain the exact sequence
\begin{equation}
    0 \to \check{H}^2_\infty(T_\alpha, \R^\delta) \to \check{H}^2_{\infty, \nabla}(T_\alpha, \R) \to \Omega^3_{\cl}(T_\alpha) \to \check{H}^3_{\infty}(T_\alpha, \R^\delta),
\end{equation}
but $\Omega^3_{\cl}(T_\alpha) \cong H^3_{\dR}(T_\alpha) \cong 0 \cong \check{H}^3_\infty(T_\alpha, \R^\delta)$, thus $\check{H}^2_{\infty, \nabla}(T_\alpha, \R) \cong \R$. Similar reasoning proves that $\check{H}^k_{\infty, \nabla}(T_\alpha, \R) \cong 0$ for $k > 2$. Thus we have proven the following.

\begin{Th} \label{th bundle gerbes of irrational torus}
Let $T_\alpha$ denote the irrational torus, then
\begin{equation}
    \check{H}^k_{\infty, \nabla}(T_\alpha, R) \cong \begin{cases}
        \R^2, & k = 1,\\
        \R, & k = 2,\\
        0, & k > 2.
    \end{cases}
\end{equation}
\end{Th}

The reader should note that the above computations only work because the irrational torus has the property that its deRham cohomology is equal to its closed forms. This is not the case for general diffeological spaces, and therefore Corollary \ref{cor prePIZ exact sequence} is not generally helpful for computations. Therefore we desire an exact sequence which uses the deRham cohomology of a diffeological space rather than its closed forms.

\begin{Def}
Given a diffeological space $X$ and $k \geq 1$, let $\check{H}^k_{\conn}(X, \R)$ denote the subspace of $H^k_{\dR}(X) \oplus \check{H}^k_{\infty}(X, \R)$ generated by the subset of pairs $([F], [g])$ where $F$ is the curvature form $F = d \omega^k$ of a diffeological $\R$-bundle $(k-1)$-gerbe with connection $(\omega^k,\omega^{k-1}, \, \dots \, , \omega^1, g)$.
\end{Def}

The vector space $\check{H}^k_{\conn}(X, \R)$ sits in an exact sequence
\begin{equation} \label{eq H_conn exact sequence}
    0 \to \check{H}^k_{\infty, \text{triv}}(X, \R^\delta) \to \check{H}^k_{\infty, \nabla}(X, \R) \to \check{H}^k_{\conn}(X, \R) \to 0,
\end{equation}
where $\check{H}^k_{\infty, \text{triv}}(X, \R^\delta) \subset \check{H}^k_{\infty}(X, \R^\delta)$ is the subspace of isomorphism classes of trivial $\R$-bundle $(k-1)$-gerbes with flat connection.

Let us now define a new sequence of vector spaces
\begin{equation*} 
    \check{H}^k_{\infty}(X, \R^\delta) \xrightarrow{\alpha} \check{H}^k_{\conn}(X, \R) \xrightarrow{\beta} H^{k+1}_{\dR}(X) \xrightarrow{\gamma} \check{H}^{k+1}_{\infty}(X, \R^\delta).
\end{equation*}
The map $\alpha$ takes an isomorphism class of an $\R$-bundle $k$-gerbe with flat connection $[\omega^k, \dots, \omega^1, g]$ and gives $(0, [g])$. The map $\beta$ sends $([F], [g]) \mapsto [F]$. Finally, $\gamma$ sends $[\omega]$ to the isomorphism class of the $\R$-bundle $k$-gerbe with connection $(\omega, 0, \dots, 0)$. This map is well defined, because if $\omega' - \omega = d \tau$ for some $\tau \in \Omega^k(X)$, then $(\omega' - \omega, 0, \dots, 0) = D(\tau, 0, \dots, 0)$.

\begin{Th} \label{th degree k PIZ exact sequence}
Given a diffeological space $X$ and $k \geq 1$, the sequence of vector spaces
\begin{equation} \label{eq degree k PIZ exact sequence}
     \check{H}^k_{\infty}(X, \R^\delta) \xrightarrow{\alpha} \check{H}^k_{\conn}(X, \R) \xrightarrow{\beta} H^{k+1}_{\dR}(X) \xrightarrow{\gamma} \check{H}^{k+1}_{\infty}(X, \R^\delta)
\end{equation}
is exact.
\end{Th}

\begin{proof}
Note that $\beta \alpha = 0$. Suppose $([F],[g]) \in \check{H}^k_{\conn}(X, \R)$ and consider $\gamma \beta ([F],[g]) = [F, 0, \dots, 0]$. Since $([F],[g]) \in \check{H}^k_{\conn}(X, \R)$, there exists a $\R$-bundle $(k-1)$-gerbe with connection $(\omega^k, \dots, \omega^0, g)$ such that
\begin{equation*}
    D(\omega^k, \dots, \omega^0, g) = (F, 0, \dots, 0).
\end{equation*}
Therefore $[F,0, \dots, 0] = 0$. In fact, $\gamma \beta ([F],[g]) = 0$ if and only if there exists an $\R$-bundle $(k-1)$-gerbe with connection such that the above equation holds. This implies that $\ker \gamma = \im \beta$.

Now suppose that $\beta([F], [g]) = [F] = 0$. Then there exists a global $k$-form $\tau$ such that $d \tau = F$. Since $([F], [g]) \in \check{H}^k_{\conn}(X, \R)$, there exists an $\R$-bundle $(k-1)$-gerbe $(\omega^k, \dots, g)$ such that $d \omega^k = F$. Then $\omega^k - \tau \in \Omega^k(QX_0)$, $d (\omega^k - \tau) = 0$ and $\delta (\omega^k - \tau) = \delta \omega^k - \delta \tau = \delta \omega^k$. Therefore $(\omega^k - \tau, \omega^{k-1}, \dots, \omega^0, g)$ defines an $\R$-bundle $k$-gerbe with flat connection. Thus $([F],[g]) = (0,[g]) \in \im \alpha$. 
\end{proof}

We will refer to the exact sequence (\ref{eq degree k PIZ exact sequence}) as the \textbf{degree $k$ PIZ exact sequence}. When $k = 1$, there is an interesting additional phenomenon.

\begin{Lemma} \label{lem cohomology classes are unique in degree 1}
Let $(A,g)$ and $(A', g)$ be $\R$-bundle $0$-gerbes/diffeological principal $\R$-bundles with connection on a diffeological space $X$ with the same underlying $\R$-cocycle $g$, then $dA$ and $dA'$ are global closed $2$-forms on $X$ and their de Rham cohomology classes agree $[dA] = [dA']$.
\end{Lemma}

\begin{proof}
Since $(A, g)$ and $(A', g)$ are both $\R$-bundle $0$-gerbes with connection on $X$, this implies that for every map of plots $f_0$, we have
\begin{equation*}
    -(\delta A)_{f_0} = dg_{f_0} = -(\delta A')_{f_0}.
\end{equation*}
Now consider the form $A' - A \in \Omega^1(QX_0)$, defined plotwise by $(A' - A)_{p_0} = A'_{p_0} - A_{p_0}$. This is a global $1$-form, because for every plot map $f_0$ we have
\begin{equation*}
    (\delta (A' - A))_{f_0} = (\delta A')_{f_0} - (\delta A)_{f_0} = 0.
\end{equation*}
Thus $A' - A \in \Omega^1(X)$. Similarly $dA'$ and $dA$ are also global $2$-forms on $X$. Now $d(A' - A) \in \Omega^2(X)$ is an exact form, and $dA' - dA = d(A' - A)$. Thus $dA'$ and $dA$ represent the same de Rham cohomology class.
\end{proof}

Lemma \ref{lem cohomology classes are unique in degree 1} implies that $\check{H}^1_{\conn}(X, \R)$ is isomorphic to the subspace of $\check{H}^1_{\infty}(X, \R)$ generated by the subset of those diffeological principal $\R$-bundles that admit a connection, as every connection produces a unique cohomology class. Thus we see that $\check{H}^1_{\conn}(X, \R)$ is exactly analogous to the term ${}^d E^{1,0}_2(X)$ in (\ref{eq patricks sequence}).

Now the degree $k$ PIZ exact sequence (\ref{eq degree k PIZ exact sequence}) seems to be missing two terms compared to (\ref{eq patricks sequence}). These two terms reappear when $k = 1$, as we shall now prove. 

\begin{Th} \label{th degree 1 PIZ exact sequence}
Given a diffeological space $X$, there exists a map $\theta: H^1_{\dR}(X) \to \check{H}^1_\infty(X, \R^\delta)$ such that the sequence of vector spaces
\begin{equation} \label{eq degree 1 piz exact sequence}
0 \to H^1_{\dR}(X) \xrightarrow{\theta} \check{H}^1_{\infty}(X, \R^\delta) \xrightarrow{\alpha} \check{H}^1_{\conn}(X, \R) \xrightarrow{\beta} H^2_{\dR}(X) \xrightarrow{\gamma} \check{H}^2_{\infty}(X, \R^\delta)
\end{equation}
is exact.
\end{Th}

\begin{proof}
The sequence is exact everywhere except for $H^1_{\dR}(X)$ and $\check{H}^1_{\infty}(X, \R^\delta)$ by Theorem \ref{th degree k PIZ exact sequence}. Now recall the isomorphism $\varphi: \check{H}^0_\infty(X, [\R^\delta \to 0]) \to \check{H}^0_\infty(X, [\R \to \Omega^1_{\cl}])$ induced by the map of presheaves of chain complexes described in Example \ref{ex flat bundle gerbes} for $k = 1$. The map $\varphi$ takes an isomorphism class of a diffeological principal $\R^\delta$-bundle cocycle $[g]$ and gives the isomorphism class of the $\R$-bundle $0$-gerbe with connection $[0,g]$. Let $\theta: H^1_{\dR}(X) \to \check{H}^0_\infty(X, [\R^\delta \to 0])$ denote the map defined as follows. Let $[A] \in H^1_{\dR}(X)$ denote a cohomology class, and suppose that $A$ is a global closed $1$-form representing this class. Since it is closed, there exists an $a \in \R(QX_0)$ such that $d a = A$. Then $\delta a$ defines an $\R^\delta$ cocycle, as $d \delta a = \delta A = 0$. Let $\theta([A]) = [\delta a]$. This map is well defined, as suppose that $a, a' \in \R(QX_0)$ such that $d a = d a' = A$. Then $a - a'$ is a $\R^\delta$-coboundary between $\delta a$ and $\delta a'$ as $d(a - a') = A - A = 0$ and $\delta a - \delta a' = \delta(a - a')$, so $[\delta a] = [\delta a']$.

We have a commutative diagram
\begin{equation}
\begin{tikzcd}
	{H^1_{\dR}(X)} & {\check{H}^0_\infty(X, [\R^\delta \to 0])} & {\check{H}^1_{\conn}(X,\R)} \\
	& {\check{H}^0_\infty(X, [\R \to \Omega^1_{\cl}])}
	\arrow["\varphi", from=1-2, to=2-2]
	\arrow["\theta", from=1-1, to=1-2]
	\arrow["{{\theta'}}"', from=1-1, to=2-2]
	\arrow["\alpha", from=1-2, to=1-3]
\end{tikzcd}
\end{equation}
where $\alpha$ takes a $\R^\delta$-cocycle and considers it as a $\R$-cocycle, and $\theta'([A]) = [-A,0]$. Now $\varphi \theta = \theta'$ because $\varphi \theta ([A]) = [0, \delta a]$, $\theta'([A]) = [-A,0]$, and $(-A, 0) - (0, \delta a) = (-A, -\delta a) = (-d a, -\delta a) = D(-a)$.

Clearly $\alpha \theta = 0$. Let us show that $\im \theta = \ker \alpha$. Suppose that $[g]$ is the isomorphism class of a diffeological principal $\R^\delta$-bundle such that it is trivial as a diffeological principal $\R$-bundle. Then there exists a $\lambda \in \R(QX_0)$ such that $g = \delta \lambda$. Then $\theta([d \lambda]) = [g]$. Now let us show that $\theta$ is injective. It is enough to show that $\theta'$ is injective, as $\varphi$ is an isomorphism. Suppose that $[A]$ and $[B]$ are cohomology classes such that $\theta'([A]) = [-A,0] = [-B, 0] = \theta'([B])$. Then there exists a $\tau \in \R(QX_0)$ such that $(-A - (-B), 0) = (B - A, 0) = D \tau$, which implies that $\delta \tau = 0$, so that $\tau$ is a global $0$-form and $d \tau = B - A$. Thus $[A] = [B]$.
Thus we have proven that $\theta$ is injective. Now abuse notation and let $\check{H}^1_{\infty}(X, \R^\delta) = \check{H}^0_\infty(X, [\R^\delta \to 0])$. This proves that the above sequence is exact everywhere.
\end{proof}

Considering again the case where $X = T_\alpha$ is the irrational torus, from (\ref{eq derham of irrational torus}), (\ref{eq cech of irrational torus}) and (\ref{eq degree 1 piz exact sequence}), we obtain that 
\begin{equation}
 \check{H}^1_{\conn}(T_\alpha, \R) \cong \R,
\end{equation}
which agrees with \cite{iglesias2023vcech}. From (\ref{eq H_conn exact sequence}) and Theorem \ref{th bundle gerbes of irrational torus} we then obtain an isomorphism 
\begin{equation}
 \check{H}^1_{\infty, \text{triv}}(T_\alpha, \R^\delta) \cong \R.   
\end{equation}

Similarly, from (\ref{eq degree k PIZ exact sequence}) we obtain an exact sequence
\begin{equation*}
   \R \to \check{H}^2_{\conn}(T_\alpha, \R) \to 0 
\end{equation*}
so that $\check{H}^2_{\conn}(T_\alpha, \R)$ is either $0$ or $\R$. 

\appendix

\section{Diffeological Principal Bundles with Connection} \label{section diff bundles with connection}

In this section we show that the notion of diffeological principal $G$-bundle with connection introduced in Example \ref{ex diff bundle with connection} is equivalent to Waldorf's, given in \cite[Definition 3.2.1]{waldorf2012transgression}.

Given a diffeological space $X$, and a Lie group $G$, recall the definition of the $\infty$-stack $\Omega^1(-,\mathfrak{g})//G$ from Example \ref{ex diff bundle with connection}. The data of a map $QX \to \Omega^1(-,\mathfrak{g})//G$ is equivalent to a $G$-cocycle $g$ and a collection $A = \{ A_{p_0} \}$ of $1$-forms $A_{p_0} \in \Omega^1(U_{p_0}, \mathfrak{g})$ satisfying \begin{equation*}
    A_{p_1} = \text{Ad}_{g_{f_0}}^{-1}(f_0^* A_{p_0}) + g_{f_0}^* \mc(G).
\end{equation*}
for every plot $p_0 : U_{p_0} \to X$. We refer to such a map $QX \to \Omega^1(-, \mathfrak{g})//G$ as a \textbf{$G$-cocycle with connection}.

\begin{Def} \label{def category of cocycles with connection}
Let $\ncat{Coc}_{\nabla}(X,G)$ denote the category whose objects are $G$-cocycles with connection on $X$, and whose morphisms $h: (A, g) \to (A', g')$ are collections $h = \{ h_{p_0} \}$ of maps $h_{p_0} : U_{p_0} \to G$ such that $h$ is a morphism of $G$-cocycles in the sense of Definition \ref{def cocycle} and $A_{p_0} = \text{Ad}_{h_{p_0}}^{-1}(A'_{p_0}) + h_{p_0}^* \mc(G)$ for every plot $p_0$ of $X$. It is easy to see that this category is a groupoid.
\end{Def}

\begin{Def}
Let $\pi: P \to X$ be a diffeological principal $G$-bundle where $G$ is a Lie group. A \textbf{Waldorf connection} on $P$ is a $1$-form $\omega \in \Omega^1(P, \mathfrak{g})$ such that 
\begin{equation} \label{eq waldorf connection}
    \rho^* \omega = \text{Ad}_g^{-1}( \text{pr}^* \omega) + g^* \mc(G)
\end{equation}
where $\rho: P \times G \to P$ is the action map, and $g: P \times G \to G$ and $\text{pr} : P \times G \to P$ are the corresponding projection maps.

A morphism $f: (\omega, P) \to (\omega', P')$ of diffeological principal $G$-bundles on $X$ with Waldorf connections is a morphism of diffeological principal $G$-bundles $f: P \to P'$ such that $f^* \omega' = \omega$. Such morphisms are isomorphisms. Given a diffeological space $X$, let $\ncat{Wal}_G(X)$ denote the groupoid of diffeological principal $G$-bundles on $X$ equipped with a Waldorf connection.
\end{Def}

In \cite[Section 3]{minichiello2022diffeological}, we showed there is a functor $\ncat{Cons} : \ncat{Coc}(X,G) \to \ncat{DiffPrin}_G(X)$ that takes a $G$-cocycle $g$ and constructs a diffeological principal $G$-bundle $\ncat{Cons}(g) = \pi : P \to X$ on $X$. Furthermore, by Theorem \ref{th diff cocycle theorem}, this functor is an equivalence. Thus we need only understand how to construct a Waldorf connection from the collection $A = \{A_{p_0} \}$ of $1$-forms and vice versa. 

So let $g = \{ g_{f_0} \}$ be a fixed $G$-cocycle representing a diffeological principal $G$-bundle $\ncat{Cons}(g) = \pi: P \to X$. We wish to construct a $1$-form $\omega$ on $P$ from a $G$-cocycle with connection $A$ on $X$. The diffeological principal $G$-bundle $\ncat{Cons}(g)$ has a canonical plotwise trivialization $\varphi_{p_0} : U_{p_0} \times G \to p_0^* P$ such that if $f_0 : U_{p_1} \to U_{p_0}$ is a map of plots, then the induced map $\widetilde{f_0}: U_{p_1} \times G \to U_{p_0} \times G$ is given by $\widetilde{f_0}(x_{p_1}, h) = (f_0(x_{p_1}), g_{f_0}(x_{p_1}) \cdot h)$, where $g_{f_0} : U_{p_1} \to G$ is the component of the $G$-cocycle on $f_0$. See \cite[Section 3]{minichiello2022diffeological} for more details.

Now let $q_0: U_{q_0} \to P$ be a plot. We obtain a commutative diagram
\begin{equation}
    \begin{tikzcd}
	{U_{q_0}} \\
	& {U_{q_0} \times G} & {p_0^*P} & P \\
	&& {U_{q_0}} & X
	\arrow["\pi", from=2-4, to=3-4]
	\arrow["{p_0}"', from=3-3, to=3-4]
	\arrow[from=2-3, to=3-3]
	\arrow[from=2-3, to=2-4]
	\arrow["\lrcorner"{anchor=center, pos=0.125}, draw=none, from=2-3, to=3-4]
	\arrow["{\varphi_{p_0}}", from=2-2, to=2-3]
	\arrow["{k_{q_0}}", from=1-1, to=2-2]
	\arrow["{q_0}", curve={height=-12pt}, from=1-1, to=2-4]
	\arrow["{1_{U_{q_0}}}"', curve={height=24pt}, from=1-1, to=3-3]
	\arrow[from=2-2, to=3-3]
\end{tikzcd}
\end{equation}
where $p_0 = \pi q_0$ and $k_{q_0} : U_{q_0} \to U_{q_0} \times G$ is the unique map given by the universal property of the pullback $U_{q_0} \times G \cong p_0^* P$. Since this map is over $U_{q_0}$, we have $k_{q_0}(x_{q_0}) = (x_{q_0}, g_{q_0}(x_{q_0}))$ for a unique map $g_{q_0} : U_{q_0} \to G$. 

Now if $f_0 : U_{p_1} \to U_{p_0}$ is a map of plots, we obtain a commutative diagram
\begin{equation}
    \begin{tikzcd}
	{U_{q_1}} && {U_{q_0}} \\
	{U_{q_1} \times G} && {U_{q_0} \times G} \\
	& P
	\arrow["{k_{q_1}}"', from=1-1, to=2-1]
	\arrow["{\widetilde{f}_0}", from=2-1, to=2-3]
	\arrow["{k_{q_0}}"', from=1-3, to=2-3]
	\arrow["{f_0}", from=1-1, to=1-3]
	\arrow[from=2-1, to=3-2]
	\arrow[from=2-3, to=3-2]
\end{tikzcd}
\end{equation}
which implies that if $x_{q_1} \in U_{q_1}$, then
\begin{equation*}
    \widetilde{f}_0 k_{q_1}(x_{q_1}) = \widetilde{f}_0(x_{q_1}, g_{q_1}(x_{q_1}))= (f_0(x_{q_1}), g_{f_0}(x_{q_1}) \cdot g_{q_1}(x_{q_1})) = (f_0(x_{q_1}), g_{q_0}(f_0(x_{q_1})) = k_{q_0} f_0(x_{q_1}).
\end{equation*}
From this we obtain the equation
\begin{equation}
    g_{f_0} \cdot g_{q_1} = (g_{q_0} \circ f_0).
\end{equation}
Now suppose that $A = \{ A_{p_0} \}$ is a $G$-cocycle with connection for the fixed cocycle $g$. We wish to obtain a $1$-form on $P$. Since $P$ is a diffeological space, we can define it plotwise. Given a plot $q_0: U_{q_0} \to P$, we obtain a plot $p_0 : U_{q_0} \to X$ of the base $X$ by setting $p_0 = \pi q_0$. Thus there is a $1$-form $A_{p_0} \in \Omega^1(U_{q_0}, \mathfrak{g})$ from the $G$-cocycle with connection. Now let
\begin{equation}
    B_{q_0} = \text{Ad}_{g_{q_0}}^{-1}( A_{p_0}) + g_{q_0}^* \mc(G).
\end{equation}
Thus $B_{q_0} \in \Omega^1(U_{q_0}, \mathfrak{g})$. We wish to show that this defines a $1$-form on $P$, namely we need to check that if $f_0: U_{q_1} \to U_{q_0}$ is a map of plots of $P$, then
\begin{equation}
    f_0^* B_{q_0} = B_{q_1}.
\end{equation}
So let $f_0 : U_{q_1} \to U_{q_0}$ be such a map of plots. Then we have
\begin{equation}
\begin{aligned}
    f_0^* B_{q_0} & = f_0^* \left( \text{Ad}_{g_{q_0}}^{-1}(A_{p_0}) + g_{q_0}^* \mc(G) \right) \\
    & = \text{Ad}_{(g_{q_0} \circ f_0)}^{-1} ( f_0^* A_{p_0} ) + (g_{q_0} \circ f_0)^* \mc(G) \\
    & = \text{Ad}_{g_{q_1}}^{-1} \text{Ad}_{g_{f_0}}^{-1} ( f_0^* A_{p_0} ) + (g_{f_0} \cdot g_{q_1})^* \mc(G) \\
    & = \text{Ad}_{g_{q_1}}^{-1} \text{Ad}_{g_{f_0}}^{-1} (f_0^* A_{p_0}) + \text{Ad}_{g_{q_1}}^{-1}(g_{f_0}^* \mc(G)) + g_{q_1}^* \mc(G) \\
    & = \text{Ad}_{g_{q_1}}^{-1} \left( \text{Ad}_{g_{f_0}}^{-1}( f_0^* A_{p_0}) + g_{f_0}^* \mc(G) \right) + g_{q_1}^* \mc(G) \\
    & = \text{Ad}_{g_{q_1}}^{-1}(A_{p_1}) + g_{q_1}^* \mc(G) \\
    & = B_{q_1}.
\end{aligned}
\end{equation}
We have used the product rule for the Maurer-Cartan form
\begin{equation}
    (g \cdot h)^* \mc(G) = \text{Ad}_h^{-1}(g^* \mc(G)) + h^* \mc(G),
\end{equation}
on the fourth line above, which can easily be verified using the description of $\mc(G)$ as $g^{-1} dg$.

Thus the collection $\{ B_{q_0} \}$ defines a $1$-form $\omega \in \Omega^1(P, \mathfrak{g})$ with $\omega_{q_0} = B_{q_0}$. We must still show that $\omega$ is a Waldorf connection.

We will check the equation (\ref{eq waldorf connection}) plotwise on $P \times G$. A plot of $P \times G$ is a pair of plots $q_0 : U_{q_0} \to P$ and $h_0 : U_{q_0} \to G$, which we shall pair to form the plot $\langle q_0, h_0 \rangle : U_{q_0} \to P \times G$. Let us examine $(\rho^* \omega)_{\langle q_0, h_0 \rangle}$. This is the $1$-form $\omega_{\rho \langle q_0, h_0 \rangle}$, where $\rho : P \times G \to P$ is the action map. We can thus write $\rho \langle q_0, h_0 \rangle = q_0 \cdot h_0$, where $\cdot$ is the action of $G$ on $P$. Thus we wish to compute $\omega_{q_0 \cdot h_0}$. Looking plotwise, it is easy to see that
\begin{equation}
g_{q_0 \cdot h_0} = g_{q_0} \cdot h_0.    
\end{equation}
Thus we have
\begin{equation}
    \begin{aligned}
        \omega_{q_0 \cdot h_0} & = B_{q_0 \cdot h_0} \\
        &= \text{Ad}_{g_{q_0 \cdot h_0}}^{-1}(A_{p_0}) + g_{q_0 \cdot h_0}^* \mc(G) \\
        &= \text{Ad}_{h_0}^{-1} \text{Ad}_{g_{q_0}}^{-1}(A_{p_0}) + (g_{q_0} \cdot h_0)^* \mc(G) \\
        &= \text{Ad}_{h_0}^{-1} \text{Ad}_{g_{q_0}}^{-1}(A_{p_0}) + \text{Ad}_{h_0}^{-1} g_{q_0}^* \mc(G) + h_0^* \mc(G) \\
        & = \text{Ad}_{h_0}^{-1}(B_{q_0}) + h_0^* \mc(G).
    \end{aligned}
\end{equation}
Pulling back to $P \times G$ gives precisely the equation (\ref{eq waldorf connection}). So given a $G$-cocycle with connection $(A, g)$, let $\ncat{Cons}_{\nabla}(A,g) = (\omega, P)$ denote the diffeological principal $G$-bundle $P = \ncat{Cons}(g)$ equipped with Waldorf connection $\omega$.

Now suppose that $h: (A, g) \to (A', g')$ is a morphism of $G$-cocycles with connection on $X$. We wish to obtain a morphism of diffeological principal $G$-bundles that preserve the Waldorf connection. By \cite[Section 3]{minichiello2022diffeological}, we know that $\ncat{Cons}(h): \ncat{Cons}(g) \to \ncat{Cons}(g')$ is a map of the respective diffeological principal $G$-bundles. We need only show that $\ncat{Cons}(h)$ preserves the Waldorf connection. Let $(\omega, P) = \ncat{Cons}_{\nabla}(A, g)$ and $(\omega', P') = \ncat{Cons}_{\nabla}(A',g')$, and let $\widetilde{h} = \ncat{Cons}(h)$ denote the corresponding morphism given by the morphism $h$ of $G$-cocycles. For every plot $q_0: U_{q_0} \to P$ we obtain the following commutative diagram
\begin{equation}
    \begin{tikzcd}
	& {U_{q_0}} \\
	{U_{q_0} \times G} && {U_{q_0} \times G} \\
	P && {P'}
	\arrow["{\widetilde{h}}"', from=3-1, to=3-3]
	\arrow["{k_{q_0}}"', from=1-2, to=2-1]
	\arrow["{k'_{\widetilde{h} q_0}}", from=1-2, to=2-3]
	\arrow["{l_{q_0}}"', from=2-1, to=3-1]
	\arrow["{l'_{\widetilde{h} q_0}}", from=2-3, to=3-3]
	\arrow["{\widetilde{h_{q_0}}}"{description}, from=2-1, to=2-3]
\end{tikzcd}
\end{equation}
where $\widetilde{h_{q_0}}(x_{q_0}, g) = (x_{q_0}, h_{q_0}(x_{q_0}) \cdot g)$ for $h_{q_0} : U_{q_0} \to G$ the component of the morphism $h$ of cocycles. The above diagram also implies that
\begin{equation}
    k'_{\widetilde{h} \, q_0}(x_{q_0}) = (x_{q_0}, g'_{\widetilde{h}  q_0}(x_{q_0})) = (x_{q_0}, h_{p_0}(x_{q_0}) \cdot g_{q_0}(x_{q_0})) = \widetilde{h_{q_0}} k_{q_0}(x_{q_0)},
\end{equation}
and thus we have
\begin{equation}
    g'_{\widetilde{h} \, q_0} = h_{q_0} \cdot g_{q_0}.
\end{equation}
We wish to show that $\widetilde{h}^* \omega' = \omega$. It is therefore equivalent to show that
\begin{equation}
    (\widetilde{h}^* \omega')_{q_0} = \omega'_{\widetilde{h} q_0} = B'_{\widetilde{h} q_0} = B_{q_0} = \omega_{q_0}.
\end{equation}
Now we have
\begin{equation} \label{eq preserve waldorf connection computation}
    \begin{aligned}
        B'_{\widetilde{h} q_0} & = \text{Ad}_{g_{\widetilde{h} q_0}}^{-1}(A'_{p_0}) + {g'}_{\widetilde{h} q_0}^* \mc(G) \\
        & = \text{Ad}_{g_{q_0}}^{-1} \text{Ad}_{h_{p_0}}^{-1}(A'_{p_0}) + (h_{p_0} \cdot g_{q_0})^* \mc(G) \\
        & = \text{Ad}_{g_{q_0}}^{-1} \text{Ad}_{h_{p_0}}^{-1}(A'_{p_0}) + \text{Ad}_{g_{q_0}}^{-1}(h_{p_0}^* \mc(G)) + g_{q_0}^* \mc(G) \\
        &= \text{Ad}_{g_{q_0}}^{-1}( \text{Ad}_{h_{p_0}}^{-1}(A'_{p_0}) + h_{p_0}^* \mc(G)) + g_{q_0}^* \mc(G) \\
        &= \text{Ad}_{g_{q_0}}^{-1} (A_{p_0}) + g_{q_0}^* \mc(G) \\
        &= B_{q_0}.
    \end{aligned}
\end{equation}
Thus $\widetilde{h}: P \to P'$ preserves the Waldorf connections. In summary, we have constructed a functor $\ncat{Cons}_{\nabla} : \ncat{Coc}(X,G) \to \ncat{Wal}_G(X)$. Now we wish to show that this functor is an equivalence of groupoids.

Now let us show that if we have a Waldorf connection $\omega$ on $P$, we can obtain an $G$-cocycle with connection. Suppose that $\pi : P \to X$ is a diffeological principal $G$-bundle, and choose a fixed plotwise trivialization $\varphi$. From this we obtain a $G$-cocycle $g$. Suppose that $\omega \in \Omega^1(P, \mathfrak{g})$ is a Waldorf connection, and let $p_0 : U_{p_0} \to X$ be a plot. We obtain the commutative diagram
\begin{equation*}
    \begin{tikzcd}
	{U_{p_0} \times G} & {p_0^*P} & P \\
	& {U_{p_0}} & X
	\arrow["\pi", from=1-3, to=2-3]
	\arrow["{p_0}"', from=2-2, to=2-3]
	\arrow[from=1-2, to=2-2]
	\arrow["{\psi_{p_0}}", from=1-2, to=1-3]
	\arrow["\lrcorner"{anchor=center, pos=0.125}, draw=none, from=1-2, to=2-3]
	\arrow["{\varphi_{p_0}}", from=1-1, to=1-2]
	\arrow[from=1-1, to=2-2]
	\arrow["{\sigma_{p_0}}", curve={height=-12pt}, from=2-2, to=1-1]
\end{tikzcd}
\end{equation*}
where $\varphi_{p_0}$ is the fixed trivialization, which is a $G$-equivariant diffeomorphism over $U_{p_0}$, and $\sigma_{p_0} : U_{p_0} \to U_{p_0} \times G$ is the canonical section $\sigma_{p_0}(x_{p_0}) = (x_{p_0}, e_G)$. Let $q_0 : U_{p_0} \to P$ be given by $q_0 = \psi_{p_0} \varphi_{p_0} \sigma_{p_0}$. Suppose that $f_0 : U_{p_1} \to U_{p_0}$ is a map of plots of $X$. Then we obtain a diagram
\begin{equation*}
\begin{tikzcd}
	{U_{p_1}} && {U_{p_0}} \\
	{U_{p_1} \times G} && {U_{p_0} \times G} \\
	{p_1^*P} && {p_0^*P} \\
	& P
	\arrow["{\sigma_{p_1}}"', from=1-1, to=2-1]
	\arrow["{\widetilde{f_0}}", from=2-1, to=2-3]
	\arrow["{\sigma_{p_0}}", from=1-3, to=2-3]
	\arrow["{f_0}", from=1-1, to=1-3]
	\arrow["{\varphi_{p_1}}"', from=2-1, to=3-1]
	\arrow["{\varphi_{p_0}}", from=2-3, to=3-3]
	\arrow["{\psi_{p_1}}"', from=3-1, to=4-2]
	\arrow["{\psi_{p_0}}", from=3-3, to=4-2]
	\arrow["{\widehat{f_0}}", from=3-1, to=3-3]
\end{tikzcd}
\end{equation*}
Notice that the middle square and the bottom triangle commute, but the top square does not commute, as $\widetilde{f_0} \sigma_{p_1}(x_{p_1}) = (f_0(x_{p_1}), g_{f_0}(x_{p_1}))$ while $\sigma_{p_0}(f_0(x_{p_1})) = (f_0(x_{p_1}), e_G)$.
Thus we have
\begin{equation}
\begin{aligned}
    q_1(x_{p_1}) & = (\psi_{p_1} \varphi_{p_1} \sigma_{p_1})(x_{p_1}) \\
    & = \psi_{p_0} \varphi_{p_0} \widetilde{f_0} \sigma_{p_1} (x_{p_1}) \\
    & = \psi_{p_0} \varphi_{p_0} (f_0(x_{p_1}), g_{f_0}(x_{p_1})) \\ 
    & = (\psi_{p_0} \varphi_{p_0})\left[ (f_0(x_{p_1}), e_G) \cdot g_{f_0}(x_{p_1}) \right] \\
    & = (\psi_{p_0} \varphi_{p_0} \sigma_{p_0} f_0)(x_{p_1}) \cdot g_{f_0}(x_{p_1}) \\
    &= (q_0 \circ f_0)(x_{p_1}) \cdot g_{f_0}(x_{p_1}).
\end{aligned}
\end{equation}
where on the fourth line we used the fact that $\varphi_{p_0}$ and $\psi_{p_0}$ are $G$-equivariant.

Thus we have obtained the equation
\begin{equation}
    q_1 = (q_0 \circ f_0) \cdot g_{f_0}.
\end{equation}

Now let $A_{p_0} = \omega_{q_0}$. Note that $f_0^* A_{p_0} \neq A_{p_1}$ since $f_0$ is not a map of plots from $q_1$ and $q_0$, i.e. $(q_0 \circ f_0) \neq q_1$. Consider the equation (\ref{eq waldorf connection}) at the plot $\langle (q_0 \circ f_0), g_{f_0} \rangle : U_{q_1} \to P \times G$. Note that 
\begin{equation}
    (\rho^* \omega)_{\langle (q_0 \circ f_0), g_{f_0} \rangle} = \omega_{\rho \langle (q_0 \circ f_0), g_{f_0} \rangle} = \omega_{(q_0 \circ f_0) \cdot g_{f_0}} = \omega_{q_1} = A_{p_1},
\end{equation}
and
\begin{equation}
    \text{Ad}_{g_{f_0}}^{-1}((\text{pr}^* \omega)_{\langle (q_0 \circ f_0), g_{f_0} \rangle}) + g_{f_0}^* \mc(G) = \text{Ad}_{g_{f_0}}^{-1}(\omega_{(q_0 \circ f_0)}) + g_{f_0}^* \mc(G).
\end{equation}
Now $f_0^* A_{p_0} = f_0^* \omega_{q_0} = \omega_{(q_0 \circ f_0)}$ because $f_0$ \text{is} a plot map from $q_0$ to $(q_0 \circ f_0)$ trivially, and $\omega$ is a $1$-form on $P$. Thus we obtain equation (\ref{eq classical diffeological connection equation}), so the collection $A = \{ A_{p_0} \}$ defines a $G$-cocycle with connection on $X$.

Now if $\widetilde{h} : (\omega, P ) \to (\omega', P')$ is a map of diffeological principal $G$-bundles with Waldorf connection, we want to show that it induces a map of $G$-cocycle with connection. We know that $\widetilde{h}$ induces a map $h$ of the $G$-cocycles $g$ and $g'$ representing $P$ and $P'$ respectively, and we wish to show that $A_{p_0} = \text{Ad}_{h_{p_0}}^{-1}(A'_{p_0}) + h_{p_0}^* \mc(G)$ for every plot $p_0: U_{p_0} \to X$. We know that $\widetilde{h}^* \omega' = \omega$, which is equivalent to asking that $B'_{\widetilde{h} q_0} = B_{q_0}$. So if $p_0$ is a plot of $X$, then we obtain a plot $q_0 : U_{p_0} \to P$ in the same way as above. We obtain
\begin{equation}
    \begin{aligned}
        A_{p_0} &= \text{Ad}_{g_{q_0}}(B_{q_0} - g_{q_0}^* \mc(G))\\
        &= \text{Ad}_{g_{q_0}}(B'_{\widetilde{h} q_0} - g_{q_0}^*\mc(G)) \\
        &= \text{Ad}_{g_{q_0}}[ \text{Ad}_{g_{\widetilde{h} q_0}} (A'_{p_0}) + g_{\widetilde{h} q_0}^* \mc(G) - g_{q_0}^*\mc(G)] \\
        &= \text{Ad}_{g_{q_0}} \text{Ad}_{g_{q_0}}^{-1} \text{Ad}_{h_{p_0}}^{-1}(A'_{p_0}) + \text{Ad}_{g_{q_0}} \text{Ad}_{g_{q_0}}^{-1}(h_{p_0}^* \mc(G)) \\
        &+ \text{Ad}_{g_{q_0}}(g_{q_0}^* \mc(G)) - \text{Ad}_{g_{q_0}}( g_{q_0}^* \mc(G)) \\
        &= \text{Ad}_{h_{p_0}}^{-1}(A'_{p_0}) + h_{p_0}^* \mc(G),
    \end{aligned}
\end{equation}
where we have basically done the computation of (\ref{eq preserve waldorf connection computation}) in reverse. Thus $h$ is a morphism of $G$-cocycles with connection. 

\begin{Th} \label{th waldorf and cocycle with connections iso}
Given a diffeological space $X$ and a Lie group $G$, the functor
\begin{equation}
    \ncat{Cons}_\nabla : \ncat{Coc}_{\nabla}(X,G) \to \ncat{Wal}_G(X),
\end{equation}
is an equivalence of groupoids.
\end{Th}

\begin{proof}
This follows from combining \cite[Theorem 5.13]{minichiello2022diffeological} with the above constructions.
\end{proof}

\begin{Rem}
It should be said that when $G = \R$ or $G = S^1$, one can check that a Waldorf connection reduces to a connection $1$-form in the sense of \cite[Section 5.3]{iglesias2023vcech}, thus we have an equivalence between all three definitions of diffeological principal $G$-bundle with connection in these cases.
\end{Rem}

\begin{Rem}
There is nothing stopping one from extending the above definition to the case when $G$ is a diffeological group. In this case then $\mathfrak{g}$ should be the internal tangent space \cite{christensen2015tangent} to the diffeological group $G$ at the identity. Nothing in this Appendix depended on $G$ being a Lie group, so the whole previous discussion extends to this case. It is an interesting question to see how far one can go with this analogy. For instance, does this extended definition agree with that given in \cite[Article 8.32]{iglesias2013diffeology}? We leave this question for future work.
\end{Rem}

\section{Totalization} \label{section totalization}

Given a presheaf of chain complexes $A$ and a diffeological space $X$, we wish to compute the $\infty$-stack cohomology of $X$ with values in $A$. This is defined as the abelian group
\begin{equation}
    \check{H}^0_\infty(X, A) \coloneqq \pi_0 \R \cH(X, A).
\end{equation}

We will use the Dold-Kan correspondence to get an amenable model for the homotopy type of $\R \cH(X,A)$.

Let $C$ be a cosimplicial chain complex, whose cosimplicial degree is denoted by the chain complex $C^p$. The $q$th degree of the chain complex $C^p$ is denoted $C^{p,q}$, with differential $d : C^{p, q} \to C^{p,q-1}$. From a cosimplicial chain complex we can obtain a (mixed) double complex by applying the dual of the Dold-Kan correspondence \cite[Section 4.3]{minichiello2022diffeological} to $C^\bullet$ to obtain a $\Z_{\geq 0} \times \Z_{\geq 0}$-graded vector space with two operators $d : C^{p,q} \to C^{p,q-1}$ given by the differential of each $C^p$ and $\delta: C^{p,q} \to C^{p+1,q}$ defined as the alternating sum $\sum_{i=0}^p (-1)^i d^i$ of the coface maps of $C$, with the property that $d \delta = \delta d$. From this we can obtain an unbounded ($\Z$-graded) chain complex $K = \tot^\Z C$, with
\begin{equation}
    (\tot^{\Z} C)_k = \prod_{q - p = k} C^{p,q}
\end{equation}
and differential $D = (d - (-1)^{q - p } \delta)$. In order to obtain a non-negatively graded chain complex, we apply smart truncation to obtain $\tot C = \tau_{\geq 0} \tot^{\Z} C$.

\begin{Prop} \label{prop total complex as an end}
Given a cosimplicial chain complex $C$, we have the following isomorphism of chain complexes,
    \begin{equation}
        \int_{n \in \ncat{\Delta}} \MapCh(N \R[\Delta^n], C^n) \cong \tot C,
    \end{equation}
where $\MapCh(N \R[\Delta^n], C^n)$ is the mapping chain complex defined by (\ref{eq internal hom of chain complexes}).
\end{Prop}

\begin{proof}
For the rest of this proof only, let $\Delta^n = N \R \Delta^n$. The chain complex $E \coloneqq \int_{n \in \ncat{\Delta}} \MapCh^\Z(\Delta^n, C^n)$ is isomorphic to the equalizer
\begin{equation}
    \text{eq} \left( \prod_{[n] \in \ncat{\Delta}} \MapCh^\Z(\Delta^n, C^n) \rightrightarrows \prod_{f : [m] \to [n]} \MapCh^\Z(\Delta^m, C^n) \right).
\end{equation}
It is equipped with the usual differential of mapping chain complexes, namely $d_E: E_k \to E_{k-1}$ is the map
\begin{equation}
    d_E(\varphi) = d_{C^\bullet} \circ \varphi - (-1)^k \varphi \circ d_{\Delta^\bullet}.
\end{equation}
Thus for $k \in \Z$, an element $\varphi \in E_k$ consists of a family of degree $k$ maps $\varphi_n : \Delta^n \to C^n$, such that for every map $f : [m] \to [n]$ the following diagram commutes
\begin{equation}
   \begin{tikzcd}
	{\Delta^m} & {C^m} \\
	{\Delta^n} & {C^n}
	\arrow["f"', from=1-1, to=2-1]
	\arrow["{\varphi_m}", from=1-1, to=1-2]
	\arrow["{\varphi_n}"', from=2-1, to=2-2]
	\arrow["{C^f}", from=1-2, to=2-2]
\end{tikzcd} 
\end{equation}
and this makes sense, as pre or post-composing a degree $k$ map of chain complexes with a chain map is a degree $k$ map. 
This is equivalent to having a commutative diagram of the form
\begin{equation} \label{eq cosimplicial diagram for total complex}
\begin{tikzcd}
	{\Delta^0} & {\Delta^1} & {\Delta^2} & \dots \\
	{C^0} & {C^1} & {C^2} & \dots
	\arrow["{\varphi_0}"', from=1-1, to=2-1]
	\arrow[shift left=2, from=1-1, to=1-2]
	\arrow[shift right=1, from=1-1, to=1-2]
	\arrow[shift left=3, from=1-2, to=1-3]
	\arrow[shift right=3, from=1-2, to=1-3]
	\arrow[from=1-2, to=1-3]
	\arrow[shift left=2, from=2-1, to=2-2]
	\arrow[shift right=1, from=2-1, to=2-2]
	\arrow[shift left=3, from=2-2, to=2-3]
	\arrow[shift right=3, from=2-2, to=2-3]
	\arrow[from=2-2, to=2-3]
	\arrow[shift left=5, from=1-3, to=1-4]
	\arrow[shift right=4, from=1-3, to=1-4]
	\arrow[shift left=2, from=1-3, to=1-4]
	\arrow[shift right=1, from=1-3, to=1-4]
	\arrow[shift left=5, from=2-3, to=2-4]
	\arrow[shift right=4, from=2-3, to=2-4]
	\arrow[shift left=2, from=2-3, to=2-4]
	\arrow[shift right=1, from=2-3, to=2-4]
	\arrow["{\varphi_1}"', from=1-2, to=2-2]
	\arrow["{\varphi_2}"', from=1-3, to=2-3]
\end{tikzcd}
\end{equation}
where we have hidden the codegeneracy maps for clarity. For each $n \geq 0$, a degree $k$ map $\varphi_n : \Delta^n \to C^n$ is equivalently the data of an element $x_n$ in degree $k + n$ in $C^n$, corresponding to the top non-degenerate $n$-simplex $\iota_n \in (\Delta^n)_n$, along with an element $x_n \circ f$ in degree $k + m$ for every map $f : [m] \to [n]$. However, the diagram commuting implies that $x_n \circ f = C^f x_m$. In other words, the data of the $\{x_n \}_{n \geq 0}$ completely determine the whole diagram. Thus for $k \in \Z$, there is a bijection $E_k \cong \left( \tot^\Z C \right)_k \cong \prod_{q - p = k} C^{p,q}.$ Furthermore their differentials agree, thus defining an isomorphism $E \cong \tot^\Z C$. Since $\int_{n \in \ncat{\Delta}} \MapCh(N \R[\Delta^n], C^n) = \tau_{\geq 0} E$ and $\tot C = \tau_{\geq 0} \tot^\Z C$, they are isomorphic.
\end{proof}

\begin{Rem}
Let $d_\Map$ and $d_v$ denote the differentials $\prod_{q - p = k} C^{p,q} \to \prod_{q - p = k -1} C^{p,q}$ defined componentwise by
\begin{equation*}
    d_{\Map} = (d - (-1)^{q - p} \delta), \qquad d_v = (d + (-1)^q \delta).
\end{equation*}
The differential $d_v$ is more commonly seen for total complexes in the literature. There is an isomorphism $(\tot C, d_\Map) \cong (\tot C, d_v)$ given as follows. We wish to find isomorphisms $\psi_k : (\tot C)_k \to (\tot C)_k$ making the following diagrams commute for all $k \geq 0$
\begin{equation*}
\begin{tikzcd}
	{\prod_{q - p = k} C^{p,q}} & {\prod_{q-p=k} C^{p,q}} \\
	{\prod_{q - p = k - 1} C^{p,q}} & {\prod_{q- p = k-1}C^{p,q}}
	\arrow["{d_{\Map}}"', from=1-1, to=2-1]
	\arrow["{\psi_k}", from=1-1, to=1-2]
	\arrow["{\psi_{k-1}}"', from=2-1, to=2-2]
	\arrow["{d_v}", from=1-2, to=2-2]
\end{tikzcd}
\end{equation*}
namely we want an isomorphism of chain complexes. Let us define maps $\sigma_{p,q} : C^{p,q} \to C^{p,q}$ by
\begin{equation*}
    \sigma_{p,q} = \begin{cases}
        \text{ id} & \text{if } p \equiv 0,3 \text{ (mod 4)} \\
        - \text{id} & \text{if } p \equiv 1,2 \text{ (mod 4)}.
    \end{cases}
\end{equation*}
Then set $\psi_k = \prod_{q - p = k} \sigma_{p,q}$. This gives the desired isomorphism\footnote{We obtained the maps $\sigma_{p,q}$ by carefully following the procedure outlined in \cite{rickard2019}.}.
\end{Rem}

Let us now examine how Proposition \ref{prop total complex as an end} helps us compute $\infty$-stack cohomology for diffeological spaces. Suppose that $A'$ is a presheaf of chain complexes such that $A = \DK \, A'$ is an $\infty$-stack, and $X$ is a diffeological space. Then the $0$th $\infty$-stack cohomology of $X$ with values in $A$ is given by $\pi_0 \cH(X,A)$. Let us compute $\cH(X,A)$.
\begin{equation}
\begin{aligned} 
    \cH(X, A) & = \u{\spre(\cart)}(QX, \DK \, A') \\
    & \cong \u{\spre(\cart)}\left( \int^n \coprod_{(f_{n-1}, \dots, f_0)} yU_{p_n} \times \Delta^n, \DK \, A' \right) \\
    & \cong \int_n \prod_{(f_{n-1}, \dots, f_0)} \u{\ncat{sSet}}(\Delta^n, \u{\spre(\cart)}(yU_{p_n}, \DK \, A')) \\
    & \cong \int_n \prod_{(f_{n-1}, \dots, f_0)} \u{\ncat{sSet}}(\Delta^n, [\DK \, A'](U_{p_n})) \\
    & \cong \int_n \prod_{(f_{n-1}, \dots, f_0)} \u{\ncat{Ch}}(N \R \Delta^n, A'(U_{p_n})) \\
    & \cong \int_n \prod_{(f_{n-1}, \dots, f_0)} \DK \, \MapCh(N \R \Delta^n, A'(U_{p_n}))\\
    & \cong \DK \int_n \MapCh(N \R \Delta^n, \prod_{(f_{n-1}, \dots, f_0)} A'(U_{p_n})) \\
    & \cong \DK \; \tot A'(QX),
\end{aligned} 
\end{equation}
where the last isomorphism follows from Proposition \ref{prop total complex as an end}, and the third to last isomorphism follows from Lemma \ref{lem chain complexes hom and DK}. Thus we have proven the following.

\begin{Prop} \label{prop cohomology with coefficients in a presheaf of chain complexes}
Given a presheaf of chain complexes $A'$ such that $A = \DK A'$ is an $\infty$-stack, and $X$ a diffeological space, the $0$th $\infty$-stack cohomology of $X$ with values in $A$ is given by
\begin{equation*}
    \check{H}^0_\infty(X, A) \cong H_0( \tot A'(QX)).
\end{equation*}
\end{Prop}

Propositon \ref{prop cohomology with coefficients in a presheaf of chain complexes} allows us to get a component level description of $\infty$-stack cohomology of diffeological spaces with values in the $\infty$-stacks of interest, see Section \ref{section examples of infinity stacks}. Let us now use Proposition \ref{prop total complex as an end} to prove the following well-known folklore result.

\begin{Prop} \label{prop holim of cosimplicial chain complex}
Let $C$ be a cosimplicial chain complex, then
\begin{equation}
    \text{holim}_{n \in \ncat{\Delta}} C^n \simeq \tot C,
\end{equation}
where we are computing the homotopy limit in the category of chain complexes equipped with the projective model structure.
\end{Prop}

\begin{proof}
First let us show that every cosimplicial chain complex $C$ is Reedy fibrant. For more information about Reedy categories, see \cite[Section 14]{riehl2014categorical}.  We wish to show that the matching map $C^n \to M^n C$ is a projective fibration. To do so, it will be sufficient to show that if $A$ is a cosimplicial vector space, then the canonical map $s: A^n \to M^n A$, defined by the same limit above, is surjective. This is sufficient because limits of chain complexes are computed degreewise, and a map is a projective fibration if and only if it is surjective in all positive degrees. We follow the proof\footnote{Note that the proof given in that note has several typographical errors, which is why we chose to reproduce a full proof here.} given in \cite[Lemma 21.1]{cosimplicialmodulejardine}. Let $D_n$ denote the category whose objects are surjective maps $[n] \xrightarrow{\sigma} [k]$ where $k = n-1$ or $k = n-2$, and whose morphisms are either identities or coface maps $s^j: [n - 1] \to [n-2]$
\begin{equation*}
\begin{tikzcd}
	& {[n]} \\
	{[n-1]} && {[n-2]}
	\arrow["{s^j}", two heads, from=2-1, to=2-3]
	\arrow["{s^i}"', two heads, from=1-2, to=2-1]
	\arrow["{s^j s^i}", two heads, from=1-2, to=2-3]
\end{tikzcd}
\end{equation*}
Then by \cite[Proposition 15.2.6]{hirschhorn2009model}, 
\begin{equation}
    M^n A \cong \lim_{\sigma : [n] \twoheadrightarrow [k]} A^k.
\end{equation}
Now let us label every object of $D_n$ by either $s^i : [n] \to [n-1]$ or $\sigma : [n] \to [n-2]$, and every non-identity morphism by a pair $(s^i, s^j)$. Since $s^j s^i = s^i s^{j+1}$ for every $i \leq j$, the objects $s^j s^i : [n] \to [n-2]$ and $s^i s^{j+1} : [n] \to [n-2]$ are the same, but the morphisms $(s^i, s^j)$ and $(s^{j+1}, s^i)$ are not. We can write the above limit as the equalizer (where we are not denoting the identity maps)
\begin{equation}
    M^n A \cong \text{eq} \left( \prod_{s^i} A^{n-1} \times \prod_{\sigma} A^{n-2} \underset{\beta}{\overset{\alpha}{\rightrightarrows}} \prod_{(s^i, s^j)} A^{n-2} \right),
\end{equation}
where $\alpha$ is defined in component $(s^i, s^j)$ by $\alpha(a, a') = s^j a_i$, and $\beta$ is defined in component $(s^i, s^j)$ by $\beta(a, a') = a'_{s^j s^i}$. Since $\beta$ in component $(s^i, s^j)$ and in component $(s^{j+1}, s^i)$ are equal $a'_{s^j s^i} = a'_{s^i s^{j+1}}$ for $i \leq j$, this equalizer will be isomorphic to the subspace of $(A^{n-1})^n$ of those tuples $a = (a_0, \dots, a_{n-1})$ where $s^j a_i = s^i a_{j+1}$ for $i \geq j$. The matching map $s: A^n \to M^n A$ is then given by $s(a_0, \dots, a_{n-1}) = (s^0 a_0, \dots, s^{n-1} a_{n-1})$.

Now let us prove by induction that $s$ is surjective. In the base case, note that $(0, \dots, 0) = s(0, \dots, 0)$. Now suppose that every element $b \in M^n A$ of the form $b = (b_0, \dots, b_{j-1}, 0, \dots, 0)$ is in the image of $s$. We wish to show that every element of the form $a = (a_0, \dots, a_{j-1}, a_j, 0, \dots, 0)$ is in the image of $s$. 

First note that for $i \leq j$, we have $s^j a_i = s^i a_{j+1} = s^i 0 = 0$. Thus $s^j d^i a_i = d^i s^{j-1} a_i = 0$ for $i < j$ and $s^j d^j a_j = a_j$. Thus 
\begin{equation*}
 a - s(d^j a_j) = (a_0 - s^0 d^j a_j, \dots, a_{j-1} - s^{j-1} d^j a_j,0, \dots, 0).  
\end{equation*}
By the induction hypothesis, there exists a $c \in A^n$ such that $s(c) = a - s(d^j a_j)$. Therefore $a = s(c + d^j a_j)$. 

So we have shown that $s: A^n \to M^n A$ is surjective. This implies that $s : C^n \to M^n C$ is a projective fibration for all cosimplicial chain complexes $C$. This implies that all cosimplicial chain complexes are Reedy fibrant. 

Since $C$ is Reedy fibrant, \cite[Theorem 19.8.7]{hirschhorn2009model} implies that the totalization of $C$ computes the homotopy limit, i.e. $\text{holim}_{\ncat{\Delta}} C \simeq \int_{n \in \ncat{\Delta}} \MapCh(N \R[\Delta^n], C^n)$. Thus Proposition \ref{prop total complex as an end} proves that $\text{holim}_{\ncat{\Delta}} C \simeq \tot C$.
\end{proof}

\begin{Rem}
During the writing of this paper, the preprint \cite{arakawa2023homotopy} came out, which also proves Proposition \ref{prop holim of cosimplicial chain complex} in greater generality. However since the scope of our argument is much smaller, we believe our proof of Proposition \ref{prop holim of cosimplicial chain complex} is simpler and more direct.
\end{Rem}

\section{Proof of Theorem \ref{th hopullback diagram}} \label{section proof}

In this section we prove Theorem \ref{th hopullback diagram}. We will need several technical preliminary results.

Given a chain complex $C$, consider the chain complex $C^{\Delta^1} \cong \MapCh(N \R \Delta^1, C)$. This is the chain complex with $C^{\Delta^1}_n = C_n \oplus C_n \oplus C_{n+1}$ for $n > 0$, and with differential
\begin{equation*}
    d_n : C_n \oplus C_n \oplus C_{n+1} \to C_{n-1} \oplus C_{n-1} \oplus C_n
\end{equation*}
given by $d_n (x,y,z) = (dx,dy, dz - (-1)^n[-x + y])$.

This means that for $k = 0$, we have
\begin{equation*}
    C^{\Delta^1}_0 \cong \ker \left( C_0 \oplus C_0 \oplus C_1 \xrightarrow{d_0} 0 \oplus 0 \oplus C_0 \right)
\end{equation*}
where $d_0(x,y,z) = (0,0, dz + x - y)$. There is an isomorphism $C^{\Delta^1}_0 \cong C_0 \oplus C_1$ given by the maps
\begin{equation*}
    \sigma : C^{\Delta^1}_0 \to C_0 \oplus C_1, \qquad \sigma(x,y,z) = (x,z)
\end{equation*}
\begin{equation*}
    \tau : C_0 \oplus C_1 \to C^{\Delta^1}_0, \qquad \tau(x,z) = (x, x + dz, z).
\end{equation*}
Thus the differential $d : C^{\Delta^1}_1 \to C^{\Delta^1}_0$ is isomorphic to the map $\alpha = \sigma \circ d_1$,
\begin{equation*}
    \alpha: C_1 \oplus C_1 \oplus C_2 \to C_0 \oplus C_1, \qquad \alpha(x,y,z) = \sigma d_1(x,y,z) = \sigma(dx,dy,dz - x + y) = (dx, dz - x + y).
\end{equation*}

The map $\pi : C^{\Delta^1} \to C \oplus C$ is given in degree $k > 0$ by
\begin{equation}
 \pi_k : C_k \oplus C_k \oplus C_{k+1} \to C_k \oplus C_k, \qquad \pi_k(x,y,z) = (x,y). 
\end{equation}
It is given in degree $k = 0$ by
\begin{equation} \label{eq projection from path chain complex degree 0}
    \pi_0: C_0 \oplus C_1 \to C_0 \oplus C_0, \qquad \pi_0(x,z) = (x, x + dz).
\end{equation}

Let us now state a few model categorical results that we will need for the proof of Theorem \ref{th hopullback diagram}.

\begin{Lemma}[{\cite[Corollary 13.3.8]{hirschhorn2009model}}] \label{lem hopullback is actual pullback when fibration}
Let $\cat{C}$ be a right proper model category and let
\begin{equation*}
    \begin{tikzcd}
	A & B \\
	C & D
	\arrow[from=1-1, to=2-1]
	\arrow[from=1-1, to=1-2]
	\arrow["g"', from=2-1, to=2-2]
	\arrow["f", from=1-2, to=2-2]
	\arrow["\lrcorner"{anchor=center, pos=0.125}, draw=none, from=1-1, to=2-2]
\end{tikzcd}
\end{equation*}
be a pullback square in $\cat{C}$ such that at least one of maps $f$ or $g$ is a fibration. Then the above square is a homotopy pullback square.
\end{Lemma}

\begin{Lemma}[{\cite[Proposition 13.3.15]{hirschhorn2009model}}] \label{lem pasting law}
Let $\cat{C}$ be a right proper model category, and suppose we have a commutative diagram of the form
\begin{equation}
    \begin{tikzcd}
	A & B & C \\
	D & E & F
	\arrow[from=1-1, to=2-1]
	\arrow[from=1-1, to=1-2]
	\arrow[from=2-1, to=2-2]
	\arrow[from=1-2, to=2-2]
	\arrow[from=2-2, to=2-3]
	\arrow[from=1-2, to=1-3]
	\arrow[from=1-3, to=2-3]
\end{tikzcd}
\end{equation}
and suppose that the right hand square is a homotopy pullback square. Then the left hand square is a homotopy pullback square if and only if the outer rectangle is a homotopy pullback square.
\end{Lemma}

\begin{Lemma}[{\cite[Corollary 2.3.10]{schreiber2013differential}}] \label{lem homotopy pullback using path space fibration}
Let $\cat{C}$ be a model category, and suppose $X,Y,Z$ are fibrant objects in $\cat{C}$ and $f: X \to Z$ and $g: Y \to Z$ are maps between them. Then the homotopy pullback of $f$ and $g$ is naturally weak equivalent to the actual pullback
\begin{equation}
    \begin{tikzcd}
	{X \times^h_Z Y} & {Z^{I}} \\
	{X \times Y} & {Z \times Z}
	\arrow[from=1-1, to=2-1]
	\arrow[from=1-1, to=1-2]
	\arrow["{f \times g}"', from=2-1, to=2-2]
	\arrow[ from=1-2, to=2-2]
	\arrow["\lrcorner"{anchor=center, pos=0.125}, draw=none, from=1-1, to=2-2]
\end{tikzcd}
\end{equation}
where $Z^{I} \to Z \times Z$ is a path object for $Z$.
\end{Lemma}

\begin{Lemma}[{\cite[Proposition 3.3.16]{hirschhorn2009model}}] \label{lem fibrations of infinity stacks}
Suppose $f: X \to Y$ is a map of $\infty$-stacks on $\cart$ that is an projective fibration. Then it is a fibration in the \v{C}ech model structure.
\end{Lemma}

\begin{Def}
If $X,Y,Z$ are $\infty$-stacks on $\cart$, and the commutative diagram
\begin{equation*}
\begin{tikzcd}
	X & Y \\
	{*} & Z
	\arrow[from=1-1, to=2-1]
	\arrow[from=1-1, to=1-2]
	\arrow["z"', from=2-1, to=2-2]
	\arrow["f", from=1-2, to=2-2]
\end{tikzcd}
\end{equation*}
is a homotopy pullback square, where $* \coloneqq \Delta^0$, then we say that the sequence of maps
\begin{equation*}
    X \to Y \to Z
\end{equation*}
is a \textbf{homotopy fiber sequence}, and we call $X$ the \textbf{homotopy fiber} of $f$ at $z$, which we sometimes denote by $\text{hofib}(f)$.
\end{Def}

\begin{Lemma} \label{lem fiber sequence gives exact sequence} Let $X \to Y \to Z$ be a homotopy fiber sequence of pointed $\infty$-stacks on $\cart$, and where the morphisms preserve the points. Then the resulting sequence
\begin{equation*}
    \check{H}^0_\infty(W, X) \xrightarrow{f} \check{H}^0_\infty(W, Y) \xrightarrow{g} \check{H}^0_\infty(W, Z),
\end{equation*}
is exact.\footnote{Exact in the sense that each set $\check{H}^0_\infty(W, A)$ is pointed by the constant map $*$ to the point of $A$, and the image of $f$ is equal to the set of $x \in \check{H}^0_\infty(W,Y)$ such that $g(x) = *$, which we call the kernel of $g$.}
\end{Lemma}

\begin{proof}
This follows from the fact that $\R \cH(-,-)$ preserves homotopy pullbacks in its second factor, so a homotopy fiber sequence of $\infty$-stacks produces a homotopy fiber sequence of spaces
\begin{equation*}
    \R \cH(W, X) \to \R \cH(W, Y) \to \R \cH(W, Z)
\end{equation*}
and the long exact sequence of homotopy groups gives exactness for $\pi_0$.
\end{proof}

\begin{Prop}
Suppose that we have a commutative square
\begin{equation} \label{eq hopullback of presheaves of chain complexes}
\begin{tikzcd}
	A & B \\
	C & D
	\arrow["g"', from=1-1, to=2-1]
	\arrow["f", from=1-1, to=1-2]
	\arrow["h"', from=2-1, to=2-2]
	\arrow["k", from=1-2, to=2-2]
\end{tikzcd}
\end{equation}
of presheaves of chain complexes over $\cart$ such that
\begin{equation}\label{eq hopullback of presheaves of chain complexes 2}
    \begin{tikzcd}
	{\DK A} & {\DK B} \\
	{\DK C} & {\DK D}
	\arrow["{\DK g}"', from=1-1, to=2-1]
	\arrow["{\DK f}", from=1-1, to=1-2]
	\arrow["{\DK h}"', from=2-1, to=2-2]
	\arrow["{\DK h}", from=1-2, to=2-2]
\end{tikzcd}
\end{equation}
is a commutative diagram of $\infty$-stacks. If (\ref{eq hopullback of presheaves of chain complexes}) is a homotopy pullback square in the projective model structure on $\ncat{ChPre}(\cart)$, then (\ref{eq hopullback of presheaves of chain complexes 2}) is a homotopy pullback square in $\cH$.
\end{Prop}

\begin{proof}
If (\ref{eq hopullback of presheaves of chain complexes}) is a homotopy pullback diagram, then $A$ is weak equivalent to the actual pullback $C \times^h_D B$ of Lemma \ref{lem homotopy pullback using path space fibration}. Since both of these presheaves of chain complexes are projective fibrant, and $\DK$ is right Quillen, then $\DK A$ is weak equivalent to $\DK (C \times^h_D B) \cong \DK C \times^h_{\DK D} \DK B$. Therefore $\DK A$ is a homotopy pullback of (\ref{eq hopullback of presheaves of chain complexes 2}) in $\H$. Then by Proposition \ref{prop model topos localization preserves finite homotopy limits}, it is a homotopy pullback in $\cH$.
\end{proof}

Now that we have all the technical tools we need, we restate Theorem \ref{th hopullback diagram} for  the convenience of the reader.

\begin{reptheorem}{th hopullback diagram}    
For every $k \geq 1$, there exists a commutative diagram of $\infty$-stacks of the following form
\begin{equation}
    \begin{tikzcd}
	{*} & {\B^k \R^\delta} & {*} & {*} \\
	{*} & {\B^k_\nabla \R} & {\Omega^{k+1}_\cl} & {\Omega^{k+1}} \\
	& {\B^k \R} & {\B^k \Omega^1_{\cl}} & {\Omega^{1 \leq \bullet \leq k+1}} \\
	& {*} & {\B^{k+1} \R^\delta} & {\B^{k+1}_\nabla \R}
	\arrow[from=1-1, to=2-1]
	\arrow[""{name=0, anchor=center, inner sep=0}, from=1-1, to=1-2]
	\arrow[""{name=1, anchor=center, inner sep=0}, from=2-1, to=2-2]
	\arrow[from=1-2, to=2-2]
	\arrow[""{name=2, anchor=center, inner sep=0}, from=2-2, to=2-3]
	\arrow[""{name=3, anchor=center, inner sep=0}, from=1-2, to=1-3]
	\arrow[from=1-3, to=2-3]
	\arrow[from=2-2, to=3-2]
	\arrow[""{name=4, anchor=center, inner sep=0}, from=2-3, to=2-4]
	\arrow[from=1-4, to=2-4]
	\arrow[""{name=5, anchor=center, inner sep=0}, from=1-3, to=1-4]
	\arrow[from=2-4, to=3-4]
	\arrow[from=2-3, to=3-3]
	\arrow[""{name=6, anchor=center, inner sep=0}, from=3-2, to=3-3]
	\arrow[""{name=7, anchor=center, inner sep=0}, from=3-3, to=3-4]
	\arrow[""{name=8, anchor=center, inner sep=0}, from=4-3, to=4-4]
	\arrow[from=3-4, to=4-4]
	\arrow[from=3-2, to=4-2]
	\arrow[""{name=9, anchor=center, inner sep=0}, from=4-2, to=4-3]
	\arrow[from=3-3, to=4-3]
	\arrow["1"{description}, draw=none, from=0, to=1]
	\arrow["2"{description}, draw=none, from=3, to=2]
	\arrow["4"{description}, draw=none, from=2, to=6]
	\arrow["3"{description}, draw=none, from=5, to=4]
	\arrow["5"{description}, draw=none, from=4, to=7]
	\arrow["6"{description}, draw=none, from=6, to=9]
	\arrow["7"{description}, draw=none, from=7, to=8]
\end{tikzcd}
\end{equation}
furthermore every commutative square in this diagram is a homotopy pullback square in $\cH$.
\end{reptheorem}

\begin{Lemma} \label{lem squares 4 and 5}
The pasted square $[4|5]$, given as follows
\begin{equation}
    \begin{tikzcd}
	{\B^k_\nabla \R} & {\Omega^{k+1}} \\
	{\B^k \R} & {\Omega^{1 \leq \bullet \leq k+1}}
	\arrow[from=1-2, to=2-2]
	\arrow[from=2-1, to=2-2]
	\arrow[from=1-1, to=2-1]
	\arrow[from=1-1, to=1-2]
\end{tikzcd}
\end{equation}
is a homotopy pullback square in $\cH$.
\end{Lemma}

\begin{proof}
Let us analyze this part of the diagram as presheaves of chain complexes.
{\small
\begin{equation}
\begin{tikzcd}
	{[\R \to \Omega^1 \to \Omega^2 \to \dots \to \Omega^k]} & {[0 \to 0 \to \dots \to \Omega^{k+1}_{\text{cl}}]} & {[0 \to 0 \to 0 \to \dots \to \Omega^{k+1}]} \\
	{[\R \to 0 \to 0 \to \dots \to 0]} & {[\Omega^1 \to \Omega^2  \to \dots \to \Omega^{k+1}_{\cl}]} & {[\Omega^1 \to \Omega^2 \to \dots \to \Omega^k \to \Omega^{k+1}]}
	\arrow[ from=1-1, to=2-1]
	\arrow[from=1-1, to=1-2]
	\arrow[from=2-1, to=2-2]
	\arrow[from=2-2, to=2-3]
	\arrow[from=1-2, to=1-3]
	\arrow[from=1-3, to=2-3]
	\arrow[from=1-2, to=2-2]
\end{tikzcd}
\end{equation}
}
where the upper horizontal left hand map is $0$ except in degree $0$ where it applies the differential $d$. The lower horizontal left hand map is $d$ in degree $k+1$ and $0$ elsewhere. The rest of the maps are either degreewise inclusions or identity maps.

Let us show that the outer rectangle is a homotopy pullback diagram. Note that neither the bottom map nor the right hand map is objectwise surjective in positive degree, namely they are not fibrations in $\ncat{ChPre}(\cart)$. However we can use Lemma \ref{lem homotopy pullback using path space fibration} to compute the homotopy pullback of $\B^k \R \to \Omega^{1 \leq \bullet \leq k+1} \leftarrow \Omega^{k+1}$. Namely it is given as the actual (objectwise) pullback of the diagram
\begin{equation}
 \begin{tikzcd}
	& {(\Omega^{1 \leq \bullet \leq k+1})^{\Delta^1}} \\
	{\B^k \R \oplus \Omega^{k+1} } & {\Omega^{1 \leq \bullet \leq k+1} \oplus \Omega^{1 \leq \bullet \leq k+1}}
	\arrow[from=2-1, to=2-2]
	\arrow[from=1-2, to=2-2]
\end{tikzcd}
\end{equation}

Now $(\Omega^{1 \leq \bullet \leq k+1})^{\Delta^1}$ is given by the presheaf of chain complexes
\begin{equation*}
   [\Omega^1 \oplus \Omega^1 \to \Omega^2 \oplus \Omega^2 \oplus \Omega^1 \to \dots \to \Omega^k \oplus \Omega^k \oplus \Omega^{k-1} \to \Omega^{k+1} \oplus \Omega^k]
\end{equation*}
which projects to $\Omega^{1 \leq \bullet \leq k+1} \oplus \Omega^{1 \leq \bullet \leq k+1}$. From this we obtain the following diagram the following diagram of presheaves of chain complexes
{\small
\begin{equation} \label{eq computation for hopullback squares 4 and 5}
    \begin{tikzcd}
	{[\R \to \Omega^1 \to \Omega^2 \to \dots \to \Omega^k]} & {[\Omega^1 \oplus \Omega^1 \to\Omega^2 \oplus \Omega^2 \oplus \Omega^1 \to \dots \to \Omega^k \oplus \Omega^k \oplus \Omega^{k-1} \to \Omega^{k+1} \oplus \Omega^k]} \\
	{[\R \to 0 \to \dots \to 0 \to \Omega^{k+1}]} & {[\Omega^1 \oplus \Omega^1 \to \Omega^2 \oplus \Omega^2 \to \dots \to \Omega^k \oplus \Omega^k \to \Omega^{k+1} \oplus \Omega^{k+1}]}
	\arrow["\pi", from=1-2, to=2-2]
	\arrow[from=2-1, to=2-2]
	\arrow[from=1-1, to=2-1]
	\arrow[from=1-1, to=1-2]
\end{tikzcd}
\end{equation}
}
and this is an actual pullback square. To see this, note that pullbacks of chain complexes are computed degreewise. For degrees $k > 0$, it is clearly a pullback. In degree $0$ we are trying to show that $\Omega^k$ is isomorphic to the pullback of
\begin{equation} \label{eq pullback using path spaces important idea}
    \begin{tikzcd}
	& {\Omega^{k+1}\oplus \Omega^k} \\
	{0 \oplus \Omega^{k+1}} & {\Omega^{k+1} \oplus \Omega^{k+1}}
	\arrow["{\pi_0}", from=1-2, to=2-2]
	\arrow["{0 \oplus 1_{\Omega^{k+1}}}"', from=2-1, to=2-2]
\end{tikzcd}
\end{equation}
but from (\ref{eq projection from path chain complex degree 0}), we know that $\pi_0(x,z) = (x,x + dz)$. For every cartesian space $U$, the pullback is the set of triples $(w, x, z) \in \Omega^{k + 1}(U) \oplus \Omega^{k+1}(U) \oplus \Omega^k(U)$ such that $x = 0$ and $w = dz$. This set is of course in bijection with $\Omega^k(U)$. Thus $[4|5]$ is a homotopy pullback square.
\end{proof}

\begin{Lemma} \label{lem square 5}
The square $[5]$, given by
\begin{equation}
    \begin{tikzcd}
	{\Omega^{k+1}_\text{cl}} & {\Omega^{k+1}} \\
	{\B^k \Omega^1_{\text{cl}}} & {\Omega^{1 \leq \bullet \leq k+1}}
	\arrow[from=1-2, to=2-2]
	\arrow[from=1-1, to=1-2]
	\arrow[from=1-1, to=2-1]
	\arrow[from=2-1, to=2-2]
\end{tikzcd}
\end{equation}
is a homotopy pullback square.
\end{Lemma}

\begin{proof}
Consider the commutative diagram of presheaves of chain complexes
\begin{equation} \label{eq square 5}
    \begin{tikzcd}
	{[0 \to 0 \to \dots \to \Omega^{k+1}_{\text{cl}}]} & {[0 \to 0 \to \dots \to \Omega^{k+1}]} \\
	{[\Omega^1 \xrightarrow{d} \Omega^2 \to \dots \to \Omega^{k+1}_{\text{cl}}]} & {[\Omega^1 \xrightarrow{d} \Omega^2 \to \dots \to \Omega^k \to \Omega^{k+1}]}
	\arrow[from=1-1, to=2-1]
	\arrow[from=2-1, to=2-2]
	\arrow[from=1-1, to=1-2]
	\arrow[from=1-2, to=2-2]
\end{tikzcd}
\end{equation}
Now the above diagram is an actual pullback, and the bottom map is objectwise a surjection in positive degrees, thus it is a fibration of presheaves of chain complexes, and therefore by Lemma \ref{lem hopullback is actual pullback when fibration} the diagram (\ref{eq square 5}) is a homotopy pullback.
\end{proof}

\begin{Cor} \label{cor square 4}
The square $[4]$ is a homotopy pullback square.
\end{Cor}

\begin{proof}
By Lemma \ref{lem squares 4 and 5}, $[4|5]$ is a homotopy pullback square. By Lemma \ref{lem square 5}, $[5]$ is a homotopy pullback square. Thus by Lemma \ref{lem pasting law}, $[4]$ is a homotopy pullback square.
\end{proof}

\begin{Lemma}
The square $[6]$
\begin{equation}
\begin{tikzcd}
	{\B^k \R} & {\B^k \Omega^1_{\cl}} \\
	{*} & {\B^{k+1} \R^{\delta}}
	\arrow[from=1-1, to=2-1]
	\arrow[from=2-1, to=2-2]
	\arrow[from=1-1, to=1-2]
	\arrow[from=1-2, to=2-2]
\end{tikzcd}
\end{equation}
is a homotopy pullback square.
\end{Lemma}

\begin{proof}
This proof is very similar as the proof of Lemma \ref{lem squares 4 and 5}. We take the actual pullback of the diagram
\begin{equation}
    \begin{tikzcd}
	& {(\B^{k+1} \R^\delta)^{\Delta^1}} \\
	{0 \oplus \B^k \Omega^1_{\cl}} & {(\B^{k+1} \R^\delta \oplus \B^{k+1} \R^\delta)}
	\arrow[from=1-2, to=2-2]
	\arrow[from=2-1, to=2-2]
\end{tikzcd}
\end{equation}
which by similar reasoning to the paragraph after (\ref{eq pullback using path spaces important idea}) is precisely
{\small
\begin{equation}
\begin{tikzcd}
	{[0 \to \Omega^1 \oplus \R \xrightarrow{d_1} \Omega^2 \oplus \Omega^1 \xrightarrow{d_2} \dots \xrightarrow{d_{k-2}} \Omega^{k-1} \oplus \Omega^{k-2} \xrightarrow{\alpha}  \Omega^{k-1}]} & {[\R \oplus \R \to \Omega^1 \oplus \Omega^1 \oplus \R \to \dots \to \Omega^k_{\cl} \oplus \Omega^{k-1}]} \\
	{[0 \to \Omega^1 \to \dots \to \Omega^k_{\cl}]} & {[\R \oplus \R \to \Omega^1 \oplus \Omega^1 \to \dots \to \Omega^k_{\cl} \oplus \Omega^k_{\cl}]}
	\arrow["\pi", from=1-2, to=2-2]
	\arrow[from=1-1, to=1-2]
	\arrow[from=1-1, to=2-1]
	\arrow[from=2-1, to=2-2]
\end{tikzcd}
\end{equation}
}
where $d_i(a,b) = (da, db -(-1)^i a)$ for $1 \leq i \leq k-2$, and $\alpha(a,b) = db + a$. Now there is an obvious map
\begin{equation}
    [0 \to \R \to \dots \to 0] \to [0 \to \Omega^1 \oplus \R \xrightarrow{d_1} \Omega^2 \oplus \Omega^1 \xrightarrow{d_2} \dots \xrightarrow{d_{k-2}} \Omega^{k-1} \oplus \Omega^{k-2} \xrightarrow{\alpha}  \Omega^{k-1}]
\end{equation}
that is an isomorphism on cohomology. Indeed $\alpha$ is surjective, and the kernel of $d_i: \Omega^i \oplus \Omega^{i -1} \to \Omega^{i+1} \oplus \Omega^i$ is the set of pairs $(a,b)$ where $a = (-1)^i db$, and these are all in the image of $d_{i-1}$.
\end{proof}

\begin{Lemma}
The square $[7]$
\begin{equation}
\begin{tikzcd}
	{\B^k \Omega^1_{\cl}} & {\Omega^{1 \leq \bullet \leq k+1}} \\
	{\B^{k+1} \R^\delta} & {\B^{k+1}_\nabla \R}
	\arrow[from=1-1, to=1-2]
	\arrow[from=1-1, to=2-1]
	\arrow[from=2-1, to=2-2]
	\arrow[from=1-2, to=2-2]
\end{tikzcd}
\end{equation}
is a homotopy pullback square.
\end{Lemma}

\begin{proof}
As presheaves of chain complexes we have
\begin{equation}
\begin{tikzcd}
	{[0 \to \Omega^1 \to \Omega^2 \to \dots \to \Omega^{k+1}_{\cl}]} & {[0 \to \Omega^1 \to \Omega^2 \to \dots \to \Omega^k \to \Omega^{k+1}]} \\
	{[\R \to \Omega^1 \to \Omega^2 \to \dots \to \Omega^{k+1}_{\cl}]} & {[\R \to \Omega^1 \to \Omega^2 \to \dots \to \Omega^k \to \Omega^{k+1}]}
	\arrow[from=1-2, to=2-2]
	\arrow[from=2-1, to=2-2]
	\arrow[from=1-1, to=2-1]
	\arrow[from=1-1, to=1-2]
\end{tikzcd}
\end{equation}
which is an actual pullback, and the bottom horizontal map is a fibration, thus by Lemma \ref{lem hopullback is actual pullback when fibration}, $[7]$ is a homotopy pullback.
\end{proof}

\begin{Lemma} \label{lem squares 2 and 4}
The pasted square $[\frac{2}{4}]$, 
\begin{equation}
    \begin{tikzcd}
	{\B^k \R^\delta} & {*} \\
	{\B^k \R} & {\B^k \Omega^1_{\text{cl}}}
	\arrow[from=1-1, to=1-2]
	\arrow[from=1-1, to=2-1]
	\arrow[from=2-1, to=2-2]
	\arrow[from=1-2, to=2-2]
\end{tikzcd}
\end{equation}
is a homotopy pullback square.
\end{Lemma}

\begin{proof}
We use the same proof technique as in Lemma \ref{lem squares 4 and 5}, namely we will compute the actual pullback of the diagram
\begin{equation}
    \begin{tikzcd}
	& {(\B^k \Omega^1_{\text{cl}})^{\Delta^1}} \\
	{\B^k \R \oplus 0} & {\B^k \Omega^1_{\text{cl}} \oplus \B^k \Omega^1_{\text{cl}} }
	\arrow[from=2-1, to=2-2]
	\arrow[from=1-2, to=2-2]
\end{tikzcd}
\end{equation}
The actual pullback we obtain is given by
{\small
\begin{equation}
       \begin{tikzcd}
	{[\R \to \Omega^1 \to \Omega^2 \to \dots \to \Omega^k_{\text{cl}}]} & {[\Omega^1 \oplus \Omega^1 \to\Omega^2 \oplus \Omega^2 \oplus \Omega^1 \to \dots \to \Omega^k_{\cl} \oplus \Omega^k_{\cl} \oplus \Omega^{k-1} \to 0 \oplus \Omega^k_{\cl}]} \\
	{[\R \to 0 \to \dots \to 0 \to 0]} & {[\Omega^1 \oplus \Omega^1 \to \Omega^2 \oplus \Omega^2 \to \dots \to \Omega^k_{\cl} \oplus \Omega^k_{\cl} \to 0]}
	\arrow["\pi", from=1-2, to=2-2]
	\arrow[from=2-1, to=2-2]
	\arrow[from=1-1, to=2-1]
	\arrow[from=1-1, to=1-2]
\end{tikzcd} 
\end{equation}
}
which is similar to the computation (\ref{eq computation for hopullback squares 4 and 5}). Thus $[\frac{2}{4}]$ is a homotopy pullback square.
\end{proof}

\begin{Cor}
The square $[2]$ is a homotopy pullback square.
\end{Cor}

\begin{proof}
By Corollary \ref{cor square 4}, Lemma \ref{lem squares 2 and 4} and Lemma \ref{lem pasting law}.
\end{proof}

\begin{Lemma} \label{lem square 3}
The square $[3]$ is a homotopy pullback square.
\end{Lemma}

\begin{proof}
As a diagram of presheaves of chain complexes
\begin{equation}
    \begin{tikzcd}
	{[0 \to \Omega^1 \to \dots \to \Omega^k]} & {[\R \to \Omega^1 \to \dots \to \Omega^k]} \\
	{[0 \to 0 \to\dots \to 0]} & {[\R \to 0 \to \dots \to 0]}
	\arrow[from=1-1, to=1-2]
	\arrow[from=2-1, to=2-2]
	\arrow[from=1-1, to=2-1]
	\arrow[from=1-2, to=2-2]
\end{tikzcd}
\end{equation}
it is an actual pullback, and the right hand map is a fibration.
\end{proof}

\begin{Lemma}
The square $[1]$ is a homotopy pullback square.
\end{Lemma}

\begin{proof}
As a diagram of presheaves of chain complexes
\begin{equation}
    \begin{tikzcd}
	{[0 \to \dots \to 0]} & {[\R \to \dots \to \Omega^k_{\cl}]} \\
	{[0 \to \dots \to 0]} & {[\R \to \dots \to \Omega^k]}
	\arrow[from=1-1, to=2-1]
	\arrow[from=2-1, to=2-2]
	\arrow[from=1-1, to=1-2]
	\arrow[from=1-2, to=2-2]
\end{tikzcd}
\end{equation}
it is an actual pullback, and the right hand map is a fibration.
\end{proof}

Thus we have proven Theorem \ref{th hopullback diagram}.

\printbibliography

\end{document}